\documentclass[10pt]{article}
\usepackage{amsthm}
\usepackage{amsmath}
\usepackage{amssymb}
\usepackage{makeidx}

\theoremstyle{definition}
\newtheorem{definition}{Definition}[section]
\newtheorem{notation}[definition]{Notation}
\newtheorem{example}[definition]{Example}
\newtheorem{openproblem}[definition]{Open problem}
\newtheorem{openquestion}[definition]{Open question}

\newtheorem{remark}[definition]{Remark}
\newtheorem{note}[definition]{Note}

\theoremstyle{plain}
\newtheorem{theorem}[definition]{Theorem}
\newtheorem{lemma}[definition]{Lemma}
\newtheorem{proposition}[definition]{Proposition}
\newtheorem{corollary}[definition]{Corollary}

\newcommand{\beq}{\begin{equation}}

\newcommand{\eeq}{\end{equation}}

\newcommand{\bdfn}{\begin{definition}}

\newcommand{\edfn}{\end{definition}}

\newcommand{\bthm}{\begin{theorem}}

\newcommand{\ethm}{\end{theorem}}

\newcommand{\bprop}{\begin{proposition}}

\newcommand{\eprop}{\end{proposition}}

\newcommand{\bcor}{\begin{corollary}}

\newcommand{\ecor}{\end{corollary}}

\newcommand{\blem}{\begin{lemma}}

\newcommand{\elem}{\end{lemma}}

\newcommand{\bex}{\begin{example}}

\newcommand{\eex}{\end{example}}

\newcommand{\bxc}{\begin{exercise}}

\newcommand{\exc}{\end{exercise}}

\newcommand{\bntn}{\begin{notation}}

\newcommand{\entn}{\end{notation}}

\newcommand{\be}{\begin{enumerate}}

\newcommand{\ee}{\end{enumerate}}

\newcommand{\bce}{\begin{center}}

\newcommand{\ece}{\end{center}}

\newcommand{\bi}{\begin{itemize}}

\newcommand{\ei}{\end{itemize}}

\newcommand{\bt}{\begin{tabular}}

\newcommand{\et}{\end{tabular}}

\newcommand{\ba}{\begin{array}} 

\newcommand{\ea}{\end{array}}

\numberwithin{equation}{section}

\def\N{{\mathbb N}}

\newcommand {\bua} {\begin{eqnarray*}}

\newcommand {\eua} {\end {eqnarray*}}

\begin{document}
\title{Boolean Lifting Property for Residuated Lattices}
\author{George GEORGESCU and Claudia MURE\c{S}AN\\ \footnotesize University of Bucharest\\ \footnotesize Faculty of Mathematics and Computer Science\\ \footnotesize Academiei 14, RO 010014, Bucharest, Romania\\ \footnotesize Emails: georgescu.capreni@yahoo.com; c.muresan@yahoo.com, cmuresan@fmi.unibuc.ro}
\date{\today }
\maketitle

\begin{abstract} In this paper we define the Boolean Lifting Property (BLP) for residuated lattices to be the property that all Boolean elements can be lifted modulo every filter, and study residuated lattices with BLP. Boolean algebras, chains, local and hyperarchimedean residuated lattices have BLP. BLP behaves interestingly in direct products and involutive residuated lattices, and it is closely related to arithmetic properties involving Boolean elements, nilpotent elements and elements of the radical. When BLP is present, strong representation theorems for semilocal and maximal residuated lattices hold.\end{abstract}

\section{Introduction}

The notion of residuated lattice originates both in algebra and in logic. It captures the structure of the set of ideals of a ring (\cite{wadi}), but also algebraic properties that come from the syntax of certain non--classical propositional logics: intuitionistic logic (\cite{hey}), \L ukasiewicz logic, BL created by Hajek (\cite{haj}), substructural logics (\cite{gal}) and other logical systems. These are the main sources that inspire the research on residuated lattices.

In ring theory, lifting idempotent property (LIP: idempotents can be lifted modulo every left (respectively right) ideal) is intensely studied (\cite{nic}, \cite{andcam}, \cite{camyu}, \cite{lam}). For instance, in \cite{nic} it is proven that a ring has LIP iff it is an exchange ring, that any clean ring has LIP, and, in the particular case of rings with central idempotents, the two notions coincide.

A lifting property for Boolean elements appears in the study of maximal MV--algebras (\cite{figele}) and maximal residuated lattices (\cite{eu3}; see also \cite{leo}). The lifting property for Boolean elements modulo the radical plays an essential part in the structure theorem for maximal residuated lattices (\cite[Theorems $6.5$ and $6.6$]{eu3}).

The purpose of the present paper is to study residuated lattices (more precisely commutative integral bounded residuated lattices) which satisfy the lifting property of Boolean elements modulo every filter, a property that we have called, for brevity, BLP. In residuated lattices, Boolean elements play a role which is similar to that of idempotent elements in rings. Certain results in this article which refer to residuated lattices with BLP are formulated analogously to properties of rings with LIP, while other results in our paper differ essentially from those found in the case of rings in both formulation and proof.

The class of residuated lattices with BLP has many important subclasses, such as the ones of Boolean algebras, chains, local and hyperarchimedean residuated lattices. It turns out that residuated lattices with BLP are exactly the quasi--local residuated lattices, which have been introduced in \cite{eu4} as a generalization for quasi--local MV--algebras (\cite{figg}) and quasi--local BL--algebras (\cite{adinggll}). Residuated lattices with BLP also coincide with those residuated lattices whose lattice of filters is dually B--normal.

The study of residuated lattices also leads to new properties, with no correspondent for rings with LIP; for instance, see, in Section \ref{clsBLP}, the properties which we have denoted by $(\star )$ and $(\star \star )$, the first of which is sufficient and the second of which is necessary for BLP to hold. These properties involve elements of the Boolean center, the radical and the set of nilpotents, and open a wide range of ways to approach the study of residuated lattices with BLP.

Finally, remarkable representation theorems hold for semilocal and for maximal residuated lattices with BLP, which also lead to important representation theorems in the particular case of MV--algebras and that of BL--algebras.

Section \ref{preliminaries} of our article consists of previously known concepts and results about residuated lattices which are necessary in the following sections. In Section \ref{BLP}, we define the BLP for residuated lattices and provide several immediate results related to this property, as well as some examples. Section \ref{charBLP} contains an arithmetic characterization of the BLP, which produces many results and further examples concerning the BLP. In Section \ref{prods}, we analyse the BLP in direct products of residuated lattices. In Section \ref{clsBLP} we set the BLP in relation to two new arithmetic conditions, we establish relationships between the classes of the local, semilocal, maximal, quasi--local, semiperfect residuated lattices, residuated lattices with BLP, we obtain representation theorems for semilocal and maximal residuated lattices with BLP, and provide a method for obtaining examples of residuated lattices with BLP and residuated lattices without BLP.

\section{Preliminaries}
\label{preliminaries}

In this section we shall present a series of known notions and results related to residuated lattices, all of which we will be using in the sequel.

We make the usual convention: throughout this paper, every algebraic structure will be designated by its support set, whenever it is clear which algebraic structure on that set we are referring to. We shall denote by $\N $ the set of the natural numbers and by $\N ^*$ the set of the nonzero natural numbers. The cardinality of a set $M$ will be denoted by $|M|$.

\begin{definition} A {\em commutative integral bounded residuated lattice} (in brief, a {\em residuated lattice}) is an algebraic structure $(A,\vee ,\wedge ,\odot ,\rightarrow ,0,1)$, where $\vee ,\wedge ,\odot ,\rightarrow $ are binary operations on $A$ and $0,1\in A$, such that $(A,\vee ,\wedge ,0,1)$ is a bounded lattice, whose partial order will be denoted $\leq $, $(A,\odot ,1)$ is a commutative monoid and the following property, called {\em the law of residuation}, is satisfied: for all $a,b,c\in A$, $a\leq b\rightarrow c$ iff $a\odot b\leq c$. The operations $\odot $ and $\rightarrow $ are called {\em multiplication} and {\em implication} (or {\em residuum}), respectively.\end{definition}

Since any residuated lattice contains the constants $0$ and $1$, it follows that each residuated lattice is non--empty. The residuated lattice with only one element (that is with $0=1$) is called {\em the trivial residuated lattice.} Residuated lattices with at least two elements (that is with $0\neq 1$) are said to be {\em non--trivial.} 

Morphisms of residuated lattices will be called, in brief, {\em residuated lattice morphisms}.

For any residuated lattice $A$ and any $a,b\in A$, we denote $a\leftrightarrow b=(a\rightarrow b)\wedge (b\rightarrow a)$ (the {\em equivalence}, or {\em biresiduum}) and $\neg \, a=a\rightarrow 0$ (the {\em negation}).

Let $A$ be a residuated lattice, $a\in A$ and $n\in \N ^{*}$, arbitrary. We shall denote by $a^{n}$ \index{$a^{n}$} the following element of $A$: $\underbrace{\textstyle a\odot \ldots \odot a}_{\textstyle n\ {\rm of}\ a}$. We also denote $a^{0}=1$.

In order to avoid excessive use of parantheses, we shall apply the usual convention: out of the previously described operations, the highest priority goes to constants and variables, followed by exponentiation, then negation, and the lowest priority goes to binary operations, including the one involving elements and filters in the notation of classes of elements modulo filters that we shall see below.

Notice that the law of residuation implies that, in any residuated lattice $A$, the implication satisfies this property: for every $x,y\in A$, $x\rightarrow y=\max \{t\in A\ |\ x\odot t\leq y\}$. Moreover, if a set $A$ is endowed with a bounded lattice structure $(A,\vee ,\wedge ,0,1)$ and an ordered commutative monoid structure $(A,\odot ,1)$ with respect to the order $\leq $ of the bounded lattice $(A,\vee ,\wedge ,0,1)$ (that is a commutative monoid structure $(A,\odot ,1)$ with the operation $\odot $ increasing in both arguments with respect to $\leq $), such that, for every $x,y\in A$, there exists $\max \{t\in A\ |\ x\odot t\leq y\}$ in $A$, then $(A,\vee ,\wedge ,\odot ,\rightarrow ,0,1)$, where the implication is defined by: $x\rightarrow y=\max \{t\in A\ |\ x\odot t\leq y\}$, for every $x,y\in A$, is a residuated lattice. We shall call this definition for the implication {\em the canonical definition of $\rightarrow $ from $\odot $}. Also, notice that, if $(A,\vee ,\wedge ,0,1)$ is a bounded lattice, then $(A,\wedge ,1)$ is an ordered commutative monoid with respect to the order relation of the bounded lattice $(A,\vee ,\wedge ,0,1)$.

It is straightforward that, given a Boolean algebra $(A,\vee ,\wedge ,\bar{\ },0,1)$, if we define on A $\odot =\wedge $ and $\rightarrow $ equal to the Boolean implication ($a\rightarrow b=\bar{a}\vee b$, for all $a,b\in A$), then the structure $(A,\vee ,\wedge ,\odot ,\rightarrow ,0,1)$ is a residuated lattice, in which $\neg \, =\bar{\ }$.

\begin{example} Let $A$ be an arbitrary bounded chain, organized as a lattice in the usual way: $\vee =\max $ and $\wedge =\min $, let $\odot =\wedge =\min $ and let us define the implication on $A$ in the following way: for every $a,b\in A$, $a\rightarrow b=\begin{cases}1, & \mbox{if }a\leq b,\\ b, & \mbox{if }a>b.\end{cases}$ Then, for every $a,b\in A$, $a\rightarrow b=\max \{x\in A\ |\ a\odot x\leq b\}$. Therefore $(A,\vee =\max ,\wedge =\min ,\odot =\wedge =\min ,\rightarrow ,0,1)$ is a residuated lattice.\label{exutil}\end{example}

The underlying lattice of $A$ is not always distributive. However, $(A,\vee ,\odot ,0,1)$ is always a bounded distributive lattice. Whenever the lattice reduct $(A,\vee ,\wedge )$ of $A$ is distributive, $A$ will be called a {\em distributive residuated lattice.} Obviously, if a residuated lattice $A$ is not distributive, then $\odot \neq \wedge $ in $A$.

Later in this paper we shall see examples of residuated lattices, and even finite distributive residuated lattices, which are neither chains, nor Boolean algebras.

Throughout the rest of this section, unless mentioned otherwise, $A$ will be an arbitrary residuated lattice.

\begin{lemma}{\rm \cite{bal}, \cite{gal}, \cite{haj},  \cite{ior}, \cite{kow}, \cite{pic}, \cite{tur}} For any $a,b,x,y\in A$, we have:
\begin{enumerate}
\item\label{aritmlr1} $a\odot b\leq a\wedge b$, that is $a\odot b\leq a$ and $a\odot b\leq b$, hence $0\odot 0=0$; moreover, if $a\leq b$ and $x\leq y$, then $a\odot x\leq b\odot y$;
\item\label{aritmlr11} if $a\vee b=1$, then $a\odot b=a\wedge b$;
\item\label{aritmlr0} $a\odot b\leq a\rightarrow b$; 
\item\label{aritmlr3} $\neg \, 1=0$ and $\neg \, 0=1$; moreover, $\neg \, a=1$ iff $a=0$;

\item\label{aritmlr7} $a\odot \neg \, a=0$; $a\odot b=0$ iff $a\leq \neg \, b$ iff $b\leq \neg \, a$;
\item\label{aritmlr10} $a\leq \neg \, \neg \, a$ and $\neg \, \neg \, \neg \, a=\neg \, a$;
\item\label{aritmlr8} $x\leq a\rightarrow x$;
\item\label{aritmlr4} $x\odot (x\rightarrow a)\leq a$;
\item\label{aritmlr2} $a\leq b$ iff $a\rightarrow b=1$; $a=b$ iff $a\leftrightarrow b=1$;
\item\label{aritmlr5} $a\rightarrow b\leq \neg \, b\rightarrow \neg \, a$;
\item\label{aritmlr6} if $a\leq b$, then: $\neg \, b\leq \neg \, a$, $b\rightarrow x\leq a\rightarrow x$ and $x\rightarrow a\leq x\rightarrow b$;
\item\label{aritmlr12} $\neg \, (a\odot b)=a\rightarrow \neg \, b=b\rightarrow \neg \, a=\neg \, \neg \, a\rightarrow \neg \, b=\neg \, \neg \, b\rightarrow \neg \, a$;
\item\label{aritmlr13} $\neg \, (a\vee b)=\neg \, a\wedge \neg \, b$;
\item\label{aritmlr9} $(a\rightarrow b)\odot (x\rightarrow y)\leq (a\wedge x)\rightarrow (b\wedge y)$.
\end{enumerate}\label{aritmlr}\end{lemma}

We refer the reader to \cite{bal}, \cite{gal}, \cite{haj}, \cite{ior}, \cite{kow}, \cite{pic}, \cite{tur} also for the following notions and results in these preliminaries.

\begin{lemma} For all $x\in A$ and all $n\in \N ^*$, $(\neg \, x)^n\leq \neg \, x^n$.\label{uncalcul}\end{lemma}

\begin{proof} For all $x\in \N $, $0=0^n=(x\odot \neg \, x)^n=x^n\odot (\neg \, x)^n$, hence $(\neg \, x)^n\leq \neg \, x^n$, according to Lemma \ref{aritmlr}, (\ref{aritmlr1}) and (\ref{aritmlr7}).\end{proof}

A nonempty subset $F$ of $A$ is called a {\em filter of $A$} iff it satisfies the following conditions, for all $a,b\in A$:

\begin{tabular}{cl}
$(1)$ & if $a,b\in F$, then $a\odot b\in F$;\\ 
$(2)$ & if $a\in F$ and $a\leq b$, then $b\in F$.
\end{tabular}

The set of all filters of $A$ is denoted by ${\cal F}(A)$.

Clearly, if $\odot =\wedge $ in $A$, then ${\cal F}(A)$ coincides with the set of filters of the bounded lattice reduct of $A$.

Notice that any filter of $A$ contains $1$, and $\{1\}$ itself is a filter, called the {\em trivial filter} of $A$. Also, $A$ is a filter, called the {\em improper filter} of $A$. A filter $F$ of $A$ is said to be {\em non---trivial} iff $F\neq \{1\}$, and it is said to be {\em proper} iff $F\neq A$. Clearly, a filter is proper iff it does not contain the element $0$. Also, $A$ is non--trivial iff $\{1\}$ is a proper filter of $A$ iff $A$ has proper filters. Obviously, the partially ordered set $({\cal F}(A),\subseteq )$ has least element $\{1\}$ and greatest element $A$. Clearly, for every filter $F$ of $A$ and every $a,b\in A$, the following equivalences hold: $a\odot b\in F$ iff $a,b\in F$ iff $a\wedge b\in F$.

It is trivial that the intersection of any family of filters of $A$ is a filter of $A$. This and the fact that $A$ is a filter of $A$ show that, for every subset $X$ of $A$, there exists a smallest filter of $A$ which includes $X$, namely the intersection of all filters of $A$ which include $X$; this filter is denoted by $[X)$ and called {\em the filter of $A$ generated by $X$.} For every $x\in A$, $[\{x\})$ is denoted, simply, by $[x)$, and it is called {\em the principal filter of $A$ generated by $x$.} It is easy to prove that, for any $x\in A$, $[x)=\{y\in A\ |\ (\exists \, n\in \N ^{*})\, x^n\leq y\}=\{y\in A\ |\ (\exists \, n\in \N )\, x^n\leq y\}$. The set of the principal filters of $A$ is denoted by ${\cal PF}(A)$.

For all filters $F$, $G$ of $A$, we denote $[F\cup G)$ by $F\vee G$. More generally, for any family $\{F_i\ |\ i\in I\}$ of filters of $A$, we denote $[\displaystyle \bigcup _{i\in I}F_i)$ by $\displaystyle \bigvee _{i\in I}F_i$.

The general form of a principal filter and Lemma \ref{aritmlr}, (\ref{aritmlr1}), immediately show that, for any $a,b\in A$:

\begin{itemize}
\item $[a)=A=[0)$ iff there exists $n\in \N ^{*}$ such that $a^n=0$ iff there exists $n\in \N $ such that $a^n=0$, while $[a)=\{1\}=[1)$ iff $a=1$;
\item $a\leq b$ iff $[b)\subseteq [a)$;
\item $[a)\cap [b)=[a\vee b)$, $[a)\vee [b)=[a\odot b)=[a\wedge b)$ and, for any $n\in \N ^{*}$, $[a^n)=[a)$. 
\end{itemize}

$({\cal{F}}(A),\vee ,\cap ,\{1\},A)$ is a complete distributive lattice, whose order relation is $\subseteq $. The last equalities above and the trivial facts that $[1)=\{1\}$ and $[0)=A$ show that ${\cal PF}(A)$ is a bounded sublattice of the bounded lattice ${\cal{F}}(A)$, and that the bounded distributive lattice $({\cal PF}(A),\vee ,\cap ,\{1\},A)$ is isomorphic to the dual of the bounded distributive lattice $(A,\vee ,\odot ,0,1)$. Indeed, $\lambda :A\rightarrow {\cal PF}(A)$, defined by $\lambda (a)=[a)$ for all $a\in A$, is a bounded lattice isomorphism between these bounded lattices. Whenever we shall talk about either of the (bounded distributive) lattices ${\cal{F}}(A)$ and ${\cal PF}(A)$, we shall refer to these bounded lattice structures.
 
The elements $x\in A$ such that $x^n=0$ for some $n\in \N ^{*}$ are called {\em nilpotent elements.} Clearly, the element $0$ is nilpotent. We shall denote by $N(A)$ the set of nilpotent elements of $A$. By the above, for any $a\in A$ and any filter $F$ of $A$:

\begin{itemize}
\item $[a)=A=[0)$ iff $a\in N(A)$, and thus we also have:
\item if $F\cap N(A)\neq \emptyset $, then $F=A$.
\end{itemize}

The elements $x\in A$ such that $x^2=x\odot x=x$ are called {\em idempotent elements}. Clearly, if an element $x\in A$ is idempotent, then $x^n=x$ for every $x\in \N ^*$, and thus $[x)=\{y\in A\ |\ x\leq y\}$. Obviously, the only element of $A$ which is both nilpotent and idempotent is $0$. Notice that, if $A$ has $\odot =\wedge $, then all elements of $A$ are idempotent. Actually, these two conditions are equivalent: $A$ has $\odot =\wedge $ iff all elements of $A$ are idempotent (\cite{eu2}).

An element $x$ of $A$ is said to be {\em regular} iff $\neg \, \neg \, x=x$. $A$ is said to be {\em involutive} iff all its elements are regular.

The set of the complemented elements of the bounded lattice reduct of $A$ is called the {\em Boolean center of $A$} and it is denoted by ${\cal B}(A)$. Clearly, $0,1\in {\cal B}(A)$. The elements of ${\cal B}(A)$ are called {\em Boolean elements of $A$}. It is known that ${\cal B}(A)$ is a Boolean algebra with the operations induced by those of $A$, together with the complementation operation given by the negation in $A$. Also, it is straightforward that ${\cal B}(A)$ is a subalgebra of the residuated lattice $A$. Here are some more properties of the Boolean center of a residuated lattice:

\begin{lemma}{\rm \cite{bal}, \cite{gal}, \cite{haj}, \cite{ior}, \cite{kow}, \cite{pic}, \cite{tur}} For every $x\in A$ and every $e,f\in {\cal B}(A)$, we have:

\begin{enumerate}
\item\label{aritmcbool1} $e$ has a unique complement, equal to $\neg \, e$; $\neg \, \neg \, e=e$; consequently, $\neg \, e=0$ iff $e=1$;
\item\label{aritmcbool2} $e\odot f=e\wedge f$, thus $e^2=e$, hence $e^n=e$ for every $n\in \N ^*$ (all Boolean elements are idempotent; the converse is not true), and therefore $[e)=\{a\in A\ |\ e\leq a\}$;
\item\label{aritmcbool4} $x\odot e=x\wedge e$ and $\neg \, e\rightarrow x=e\vee x$;
\item\label{aritmcbool3} $x\in {\cal B}(A)$ iff $x\vee \neg \, x=1$.
\end{enumerate}\label{aritmcbool}\end{lemma}

According to \cite[Proposition $2.16$]{eu3}, for any $e\in {\cal B}(A)$, $([e),\vee ,\wedge ,\odot ,\rightarrow _e,e,1)$ is a residuated lattice, where all binary operations are induced by those of $A$, except the implication, which is defined from that of $A$ by: $a\rightarrow _eb=e\vee (a\rightarrow b)$ for all $a,b\in [e)$.

\begin{remark} A rather interesting fact is that, if the underlying bounded lattice of a residuated lattice $A$ is a Boolean algebra, then ${\cal B}(A)=A$, hence, in $A$, $\odot =\wedge $, as shown by Lemma \ref{aritmcbool}, (\ref{aritmcbool2}). This means that every residuated lattice whose bounded lattice reduct is a Boolean algebra is obtained from that Boolean algebra in the way described in Section \ref{preliminaries}, in the first example of residuated lattice.\label{interestingfact}\end{remark}

Clearly, if $L$ and $M$ are residuated lattices and $\varphi :L\rightarrow M$ is a residuated lattice morphism, then $\varphi ({\cal B}(L))\subseteq {\cal B}(M)$, hence we can define a function ${\cal B}(\varphi ):{\cal B}(L)\rightarrow {\cal B}(M)$ by: ${\cal B}(\varphi )(x)=\varphi (x)$ for every $x\in {\cal B}(L)$. It is immediate that ${\cal B}(\varphi )$ is a Boolean morphism and that the mapping ${\cal B}$ defined above is a covariant functor from the category of residuated lattices to the category of Boolean algebras.

If $L$ is a bounded distributive lattice, then the set of the complemented elements of $L$ is called the {\em Boolean center of $L$} and denoted by ${\cal B}(L)$. Clearly, ${\cal B}(L)$ is a bounded sublattice of the bounded distributive lattice $L$, and, considerred with this bounded distributive lattice structure together with the complementation operation, ${\cal B}(L)$ becomes a Boolean algebra.

The Boolean center ${\cal B}(A)$ of the residuated lattice $A$ coincides with the Boolean center of the bounded distributive lattice $(A,\vee ,\odot ,0,1)$. Indeed, as shown by Lemma \ref{aritmlr}, (\ref{aritmlr11}), an element $x\in A$ has a complement in the bounded lattice $(A,\vee ,\wedge ,0,1)$ iff $x$ has a complement in the bounded lattice $(A,\vee ,\odot ,0,1)$, and the complements of $x$ in these two bounded lattices coincide. This is actually a way to prove that, although the bounded lattice $(A,\vee ,\wedge ,0,1)$ is not always distributive, every complemented element of this bounded lattice has a unique complement.

A proper filter $P$ of $A$ is called a {\em prime filter} iff, for all $a,b\in A$, if $a\vee b\in P$, then $a\in P$ or $b\in P$. The set of all prime filters of $A$ is called {\em the (prime) spectrum of $A$} and denoted by ${\rm Spec}(A)$.

The maximal elements of the set of proper filters of $A$ (with respect to $\subseteq $) are called {\em maximal filters}. The set of all maximal filters of $A$ is called {\em the maximal spectrum of $A$} and denoted by ${\rm Max}(A)$.

\begin{lemma}{\rm \cite{bal}, \cite{gal}, \cite{haj}, \cite{ior}, \cite{kow}, \cite{pic}, \cite{tur}}
\begin{enumerate}
\item\label{maxsiprime1} Every maximal filter of $A$ is a prime filter, that is: ${\rm Max}(A)\subseteq {\rm Spec}(A)$.
\item\label{maxsiprime2} Every proper filter of $A$ is included in a maximal filter of $A$, that is: for every $F\in {\cal F}(A)$, there exists $M\in {\rm Max}(A)$ such that $F\subseteq M$. Consequently, if $A$ is non--trivial, then ${\rm Max}(A)$ is non--empty.
\item\label{maxsiprime3} Every (proper) filter of $A$ is equal to the intersection of a (non--empty) family of prime filters of $A$, that is: for every (proper) filter $F$ of $A$, there exists a (non--empty) family $(Pr_i)_{i\in I}\subseteq {\rm Spec}(A)$, such that $\displaystyle F=\bigcap _{i\in I}Pr_i$.
\end{enumerate}\label{maxsiprime}\end{lemma}

The intersection of all maximal filters of $A$ is called {\em the radical of $A$} and it is denoted by ${\rm Rad}(A)$. Clearly, ${\rm Rad}(A)=A$ if $A$ is trivial, and ${\rm Rad}(A)$ is a proper filter of $A$, if $A$ is non--trivial. An element $a\in A$ is said to be {\em dense} iff $\neg \, a=0$. The set of the dense elements of $A$ is denoted by $D(A)$: $D(A)=\{a\in A\ |\ \neg \, a=0\}$.

\begin{lemma}{\rm \cite{bal}, \cite{gal}, \cite{haj}, \cite{ior}, \cite{kow}, \cite{pic}, \cite{tur}} ${\rm Rad}(A)=\{x\in A\ |\ (\forall \, n\in \N ^*)\, (\exists \, k_n\in \N ^*)\, (\neg \, x^n)^{k_n}=0\}$.\label{radunit}\end{lemma}

\begin{corollary}
Any element $x\in {\rm Rad}(A)$ has $\neg \, x\in N(A)$.\label{radnegnil}\end{corollary}

\begin{proof}
Let $x\in {\rm Rad}(A)$ and let us take $n=1$ in Lemma \ref{radunit}. Then it follows that there exists $k\in \N ^*$ such that $(\neg \, x)^k=0$.\end{proof}

\begin{lemma}{\rm \cite{bal}, \cite{gal}, \cite{haj}, \cite{ior}, \cite{kow}, \cite{pic}, \cite{tur}} \begin{enumerate}
\item\label{boolsirad} ${\cal B}(A)\cap {\rm Rad}(A)=\{1\}$.
\item\label{dense} $D(A)$ is a filter of $A$ and $D(A)\subseteq {\rm Rad}(A)$.
\item\label{consecinta} ${\cal B}(A)\cap D(A)=\{1\}$.
\end{enumerate}\label{lemaradsid}\end{lemma}

In the previous lemma, (\ref{consecinta}) is an immediate consequence of (\ref{boolsirad}), (\ref{dense}) and the fact that any filter of $A$ contains $1$.

$A$ is said to be {\em local} iff it has exactly one maximal filter, which is equivalent to the fact that ${\rm Rad}(A)$ is a maximal filter of $A$, because of the very definition of ${\rm Rad}(A)$ and the obvious fact that the set of maximal filters of $A$ is unordered with respect to set inclusion. So $A$ is local iff ${\rm Rad}(A)\in {\rm Max}(A)$ iff ${\rm Max}(A)= \{{\rm Rad}(A)\}$. 

\begin{proposition}{\rm \cite{lciu}} The following are equivalent:

\begin{enumerate}
\item\label{caractloc1} $A$ is local;
\item\label{caractloc0} $A\setminus N(A)$ is a filter of $A$;
\item $A\setminus N(A)$ is a proper filter of $A$;
\item $A\setminus N(A)$ is a maximal filter of $A$;
\item $A\setminus N(A)$ is the only maximal filter of $A$;
\item ${\rm Rad}(A)=A\setminus N(A)$;
\item\label{caractloc2} $A=N(A)\cup {\rm Rad}(A)$;
\item for all $x,y\in A$, if $x\odot y\in N(A)$, then $x\in N(A)$ or $y\in N(A)$.\end{enumerate}\label{caractloc}\end{proposition}

\begin{lemma}{\rm \cite[p. 136]{lciu}, \cite{eu4}} If $A$ is local, then ${\cal B}(A)=\{0,1\}$ and $A=N(A)\cup \{x\in A\ |\ \neg \, x\in N(A)\}$.\label{implicdir}\end{lemma}

$A$ is said to be {\em semilocal} iff ${\rm Max}(A)$ is finite. Semilocal residuated lattices include the trivial residuated lattice, local residuated lattices, finite residuated lattices, finite direct products of local or other semilocal residuated lattices (see Section \ref{prods}) etc..

The residuated lattice $A$ is said to be {\em simple} iff it has exactly two filters, that is iff $A$ is non--trivial and ${\cal F}(A)=\{{1},A\}$. Obviously, $A$ is simple iff $\{1\}$ is a maximal filter of $A$ iff $\{1\}$ is the unique maximal filter of $A$ iff $A$ is local and ${\rm Rad}(A)=\{1\}$.

An element $a\in A$ is said to be {\em archimedean} iff $a^n\in {\cal B}(A)$ for some $n\in \N ^*$. $A$ is said to be {\em hyperarchimedean} iff all its elements are archimedean. Clearly, if ${\cal B}(A)=A$, that is if the underlying bounded lattice of $A$ is a Boolean algebra, then $A$ is a hyperarchimedean residuated lattice.

Until mentioned otherwise, let $F$ be a filter of $A$. For all $a,b\in A$, we denote $a\equiv b\, (\!\!\!\!\mod F)$ and say that {\em $a$ and $b$ are congruent modulo $F$} iff $a\leftrightarrow b\in F$. It is known and easy to check that $\equiv (\!\!\!\!\mod F)$ is a congruence relation on $A$. We shall denote the quotient set $A/_{\textstyle \equiv (\!\!\!\!\mod F)}$, simply, by $A/F$, and its elements by $a/F$, with $a\in A$. So $A/F=\{a/F\ |\ a\in A\}$, where, for every $a\in A$, $a/F=\{b\in A\ |\ a\equiv b\, (\!\!\!\!\mod F)\}$ (the congruence class of $a$ with respect to $\equiv (\!\!\!\!\mod F)$). Thus, for every $a,b\in A$, $a/F=b/F$ iff $a\leftrightarrow b\in F$. Also, we shall denote by $p_F:A\rightarrow A/F$ the canonical surjection: for all $a\in A$, $p_F(a)=a/F$, and, for every $X\subseteq A$, we shall denote $X/F=p_F(X)=\{x/F\ |\ x\in X\}$. Residuated lattices form an equational class, which ensures us that the quotient set $A/F$ of $A$ with respect to the congruence relation $\equiv (\!\!\!\!\mod F)$ is a residuated lattice, with the residuated lattice operations defined canonically from those of $A$, which makes $p_F$ a surjective residuated lattice morphism. It is easy to see, from Lemma \ref{aritmlr}, (\ref{aritmlr2}), that the residuated lattices $A$ and $A/\{1\}$ are isomorphic, while the residuated lattice $A/A$ is trivial. 

It is straightforward that, for every $a,b\in A$, we have:

\begin{itemize}
\item $1/F=F$, that is: $a/F=1/F$ iff $a\in F$; consequently, $a/F=0/F$ iff $\neg \, a\in F$;
\item $a/F\leq b/F$ iff $a\rightarrow b\in F$; consequently, if $a\leq b$ then $a/F\leq b/F$.
\end{itemize}

The mapping $G\rightarrow G/F$ sets a bijection between $\{G\in {\cal F}(A)\ |\ F\subseteq G\}$ and ${\cal F}(A/F)$. Furthermore, the mapping $M\rightarrow M/F$ sets a bijection between $\{M\in {\rm Max}(A)\ |\ F\subseteq M\}$ and ${\rm Max}(A/F)$. From this we immediately get that, when $F\subseteq {\rm Rad}(A)$,  ${\rm Max}(A/F)$ is in bijection to $\{M\in {\rm Max}(A)\ |\ F\subseteq M\}={\rm Max}(A)$, and that ${\rm Rad}(A/F)={\rm Rad}(A)/F$. Consequently, $|{\rm Max}(A/{\rm Rad}(A))|=|{\rm Max}(A)|$ and ${\rm Rad}(A/{\rm Rad}(A))={\rm Rad}(A)/{\rm Rad}(A)=\{1/{\rm Rad}(A)\}$ (see also \cite[Proposition $2.12$]{eu3}).

Here is the Second Isomorphism Theorem for residuated lattices: for all filters $F$, $G$ of $A$ such that $F\subseteq G$, the residuated lattices $A/G$ and $(A/F)/_{\textstyle (G/F)}$ are isomorphic (the residuated lattice isomorphism maps $x/G\rightarrow (x/F)/_{\textstyle (G/F)}$ for every $x\in A$).

\section{Boolean Lifting Property}
\label{BLP}

In this section we shall define the Boolean Lifting Property for residuated lattices and provide several examples related to it.

Throughout this section, unless mentioned otherwise, $A$ will be an arbitrary residuated lattice and $F$ will be an arbitrary filter of $A$.

The canonical morphism $p_F:A\rightarrow A/F$ induces a Boolean morphism ${\cal B}(p_F):{\cal B}(A)\rightarrow {\cal B}(A/F)$. The range of this Boolean morphism is ${\cal B}(p_F)({\cal B}(A))=p_F({\cal B}(A))={\cal B}(A)/F$.

\begin{lemma}
\begin{enumerate}
\item\label{lifted1} ${\cal B}(A)=\{x\in A\ |\ x\vee \neg \, x=1\}$;
\item\label{lifted2} ${\cal B}(A)/F=\{x/F\ |\ x\in A,x\vee \neg \, x=1\}$;
\item\label{lifted3} ${\cal B}(A/F)=\{x/F\ |\ x\in A,x\vee \neg \, x\in F\}$;
\item\label{lifted4} ${\cal B}(A)/F\subseteq {\cal B}(A/F)$.
\end{enumerate}\label{lifted}\end{lemma}

\begin{proof} (\ref{lifted1}) is exactly Lemma \ref{aritmcbool}, (\ref{aritmcbool3}), and (\ref{lifted2}) follows from (\ref{lifted1}).

\noindent (\ref{lifted3}) By (\ref{lifted1}), ${\cal B}(A/F)=\{x/F\ |\ x\in A,x/F\vee \neg \, (x/F)=1/F\}=\{x/F\ |\ x\in A,(x\vee \neg \, x)/F=1/F\}=\{x/F\ |\ x\in A,x\vee \neg \, x\in F\}$.

\noindent (\ref{lifted4}) follows from (\ref{lifted2}), (\ref{lifted3}) and the fact that $1\in F$.\end{proof}

\begin{definition} We say that a Boolean element $f\in {\cal B}(A/F)$ can be {\em lifted} iff there exists a Boolean element $e\in {\cal B}(A)$ such that $e/F=f$. In other words, $f\in {\cal B}(A/F)$ can be lifted iff $f\in {\cal B}(A)/F$.

We say that the filter $F$ has {\em the Boolean lifting property (BLP)} iff all Boolean elements of $A/F$ can be lifted. In other words, $F$ has BLP iff the Boolean morphism ${\cal B}(p_F):{\cal B}(A)\rightarrow {\cal B}(A/F)$ is surjective. In symbols: $F$ has BLP iff ${\cal B}(A/F)={\cal B}(A)/F$.

We say that the residuated lattice $A$ has {\em the Boolean lifting property (BLP)} iff all of its filters have BLP.\end{definition}

The definition of BLP for residuated lattices is inspired by the Lifting Idempotents Property (LIP) for rings (\cite{nic}).

\begin{remark}{\rm \cite{figele}, \cite{adinggll}, \cite{leo}} If $A$ is a BL--algebra, then ${\rm Rad}(A)$ has BLP.\end{remark}

In \cite{eu3}, we have proven a structure theorem for maximal residuated lattices whose radical has BLP. See below, in Example \ref{exfarablp}, a residuated lattice whose radical does not have BLP (in particular, this residuated lattice does not have BLP).

\begin{openquestion} Can sufficient, or even necessary and sufficient conditions be provided for a residuated lattice $A$ to be such that ${\rm Rad}(A)$ has BLP? Of course, we are referring to conditions which are not trivial (such as requesting $A$ to have BLP) and involve solely arithmetic properties, properties of the prime and maximal spectrum, and/or related to important classes of residuated lattices $A$ may need to belong to, although we welcome any characterization that would turn to be useful for a further study of this kind of residuated lattices.\end{openquestion}

\begin{note} The previous open question is certainly worth studying. This is because, concerning the study of the validity of the BLP for individual filters of a residuated lattice $A$, probably the filter with respect to which such an investigation is of the greatest importance is ${\rm Rad}(A)$. Just see below results that show cases when the fact that ${\rm Rad}(A)$ has BLP implies (so it is equivalent to) the fact that $A$ has BLP: Propositions \ref{p3star} and \ref{semiloc}, Corollaries \ref{sperfect}, \ref{lafinite}, \ref{corsemiloc}, \ref{cortare}, \ref{totcortare}, \ref{mymy} and \ref{wowwow}, Proposition \ref{resultsflow}, Corollary \ref{furtherflow}.\end{note}

\begin{proposition}
\begin{enumerate}
\item\label{filtblp1} The trivial filter and the improper filter have BLP. In the case of the trivial filter, the image of the canonical morphism through the functor ${\cal B}$ is bijective.
\item\label{filtblp2} If ${\cal B}(A/F)=\{0/F,1/F\}$, then the filter $F$ has BLP.\end{enumerate}\label{filtblp}\end{proposition}

\begin{proof} (\ref{filtblp1}) It is an immediate consequence of Lemma \ref{aritmlr}, (\ref{aritmlr2}), that $\equiv (\!\!\!\!\mod \{1\})$ is the equality between elements of $A$, hence $p_{\{1\}}:A\rightarrow A/\{1\}$ is a residuated lattice isomorphism. Then, by applying to it the functor ${\cal B}$, we get that ${\cal B}(p_{\{1\}}):{\cal B}(A)\rightarrow {\cal B}(A/\{1\})$ is a Boolean isomorphism, thus a bijection, thus a surjection, so $\{1\}$ has BLP.

Every $a\in A$ has $a/A=1/A$, hence $A/A=\{1/A\}=\{p_A(1)\}=B(A)/A$. Therefore A has BLP. This statement could also have been deduced from (\ref{filtblp2}) below.

\noindent (\ref{filtblp2}) Assume that ${\cal B}(A/F)=\{0/F,1/F\}$. Since $\{0,1\}\subseteq {\cal B}(A)$, it follows that ${\cal B}(A/F)=\{0/F,1/F\}\subseteq {\cal B}(A)/F$, hence ${\cal B}(A)/F={\cal B}(A/F)$ by Lemma \ref{lifted}, (\ref{lifted4}). This means that $F$ has BLP.\end{proof}

\begin{remark} (\ref{filtblp1}) from the previous lemma shows that, if ${\cal F}(A)=\{\{1\},A\}$, then $A$ has BLP, that is every residuated lattice that is either trivial or simple has BLP.\label{splblp}\end{remark}

\begin{remark} Any Boolean algebra induces a residuated lattice with BLP. This is because, if $A$ is a Boolean algebra, from which we obtain a residuated lattice in the usual way, then ${\cal B}(A)=A$, hence, for every filter $F$ of $A$, ${\cal B}(A)/F=A/F\supseteq {\cal B}(A/F)\supseteq {\cal B}(A)/F$, therefore ${\cal B}(A)/F={\cal B}(A/F)$, so $F$ has BLP. Also, ${\cal B}(A/F)=A/F$, that is $A/F$ is also a Boolean algebra, for every filter $F$ of $A$; this fact could have been noticed in many different ways, including straightforward calculation: since every $x\in A$ has a complement $\neg \, x\in A$, it follows that every $x/F\in A/F$ has a complement $\neg \, x/F=\neg \, (x/F)\in A/F$, thus ${\cal B}(A/F)=A/F$.\label{boolblp}\end{remark}

Next we shall provide an example of residuated lattice with BLP which is not a Boole algebra and has three filters, hence it has a filter with BLP which is both non--trivial and proper. Then we shall give a remark that generalizes this example. Afterwards, we shall provide an example of a finite distributive residuated lattice without BLP.

\begin{example} Let $A=\{0,a,1\}$ be the three--element chain ($0<a<1$), organized as a residuated lattice as in Example \ref{exutil}. Then $[a)=\{a,1\}$ is a filter of this residuated lattice which is both non--trivial and proper. $0/[a)=\{0\}$ and $a/[a)=1/[a)=[a)$, thus $A/[a)=\{0/[a),1/[a)\}$, hence ${\cal B}(A/[a))=\{0/[a),1/[a)\}$, therefore $[a)$ has BLP, by Lemma \ref{filtblp}, (\ref{filtblp2}). Actually, since ${\cal B}(A)=\{0,1\}$, with $0\neq 1$, and $0/[a)\neq 1/[a)$, it follows that ${\cal B}(p_{[a)})$ is bijective.

${\cal F}(A)=\{[0),[a),[1)\}=\{A,\{a,1\},\{1\}\}$, and $A$ and $\{1\}$ have BLP by Lemma \ref{filtblp}, (\ref{filtblp1}).

Therefore $A$ has BLP. Of course, $A$ is not a Boolean algebra.\label{coolex}\end{example}

\begin{remark} Any linearly ordered residuated lattice has BLP (regardless of its exact residuated lattice structure).

Indeed, if $A$ is a linearly ordered residuated lattice, then obviously ${\cal B}(A)=\{0,1\}$, and, for every filter $F$ of $A$, the residuated lattice $A/F$ is also linearly ordered, hence ${\cal B}(A/F)=\{0/F,1/F\}$, hence $F$ has BLP, as shown by Lemma \ref{filtblp}, (\ref{filtblp2}). Therefore $A$ has BLP.

Actually, it is easy to notice that:

\begin{itemize}
\item if $A$ is the one--element chain, then its only filter, $A$, has ${\cal B}(p_A)$ bijective;
\item if $A$ is a chain with at least two elements, then every proper filter $F$ of $A$ has ${\cal B}(p_F)$ bijective, while ${\cal B}(p_A)$ is surjective and non--injective.\end{itemize}\label{chainblp}\end{remark}

\begin{example} Consider the following example of residuated lattice from \cite{ior}: $A=\{0,a,b,c,1\}$, with the bounded lattice structure given by the Hasse diagram below, $\odot =\wedge $ and the implication given by the following table:

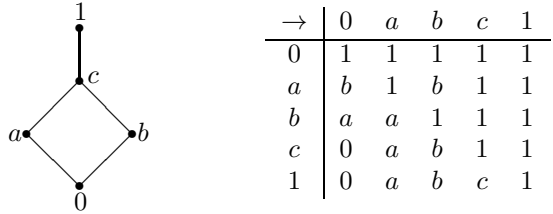
\begin{center}
\begin{tabular}{cc}
\begin{picture}(70,85)(0,0)
\put(30,10){\line(-1,1){20}}
\put(30,10){\line(1,1){20}}
\put(30,50){\line(-1,-1){20}}
\put(30,50){\line(1,-1){20}}
\put(30,50){\line(0,1){20}}
\put(30,10){\circle*{3}}

\put(10,30){\circle*{3}}
\put(50,30){\circle*{3}}
\put(30,50){\circle*{3}}
\put(30,70){\circle*{3}}
\put(28,1){$0$}
\put(3,27){$a$}
\put(52,27){$b$}
\put(33,49){$c$}
\put(28,73){$1$}
\end{picture}
&\hspace*{15pt}
\begin{picture}(70,85)(0,0)
\put(0,39){\begin{tabular}{c|ccccc}
$\rightarrow $ & $0$ & $a$ & $b$ & $c$ & $1$ \\ \hline
$0$ & $1$ & $1$ & $1$ & $1$ & $1$ \\

$a$ & $b$ & $1$ & $b$ & $1$ & $1$ \\
$b$ & $a$ & $a$ & $1$ & $1$ & $1$ \\
$c$ & $0$ & $a$ & $b$ & $1$ & $1$ \\
$1$ & $0$ & $a$ & $b$ & $c$ & $1$
\end{tabular}}
\end{picture}
\end{tabular}
\end{center}

Then ${\cal B}(A)=\{0,1\}$. Let us consider the filter $[c)=\{c,1\}$, and the element $a\notin {\cal B}(A)$. $\neg \, a=a\rightarrow 0=b$, thus $a\vee \neg \, a=a\vee b=c\in [c)$, hence $a/[c)\in {\cal B}(A/[c))$, as shown by Lemma \ref{lifted}, (\ref{lifted3}). $a\notin [c)=\{c,1\}$, thus $a/[c)\neq c/[c)$ and $a/[c)\neq 1/[c)$. $a\leftrightarrow 0=(a\rightarrow 0)\wedge (0\rightarrow a)=b\wedge 1=b\notin [c)$, hence $a/[c)\neq 0/[c)$, and $a\leftrightarrow b=(a\rightarrow b)\wedge (b\rightarrow a)=b\wedge a=0\notin [c)$, hence $a/[c)\neq b/[c)$. Therefore $a/[c)=\{a\}$.

Summarizing the above, we have: $a/[c)\in {\cal B}(A/[c))$, while $a\notin {\cal B}(A)$ and $a/[c)=\{a\}$, hence $a/[c)\notin {\cal B}(A)/[c)$, therefore ${\cal B}(A/[c))\neq {\cal B}(A)/[c)$, which means that $[c)$ does not have BLP. Hence $A$ does not have BLP. Also, notice that the maximal filters of $A$ are $[a)$ and $[b)$, hence ${\rm Rad}(A)=[c)$, thus ${\rm Rad}(A)$ does not have BLP in this example.\label{exfarablp}\end{example}

\begin{proposition}
\begin{enumerate}
\item\label{psimaublp1} Every prime filter of a residuated lattice has BLP.
\item\label{psimaublp2} Every maximal filter of a residuated lattice has BLP.
\end{enumerate}\label{psimaublp}\end{proposition}

\begin{proof} (\ref{psimaublp1}) Let $P$ be a prime filter of the residuated lattice $A$. Assume by absurdum that $P$ does not have BLP, that is ${\cal B}(A)/P\neq {\cal B}(A/P)$, which means that ${\cal B}(A)/P\subsetneq {\cal B}(A/P)$ (see Lemma \ref{lifted}, (\ref{lifted4})). This means that there exists an element $x\in A$ such that $x\in {\cal B}(A/P)$, but $x\notin {\cal B}(A)/P$. Then, according to Lemma \ref{lifted}, (\ref{lifted3}), $x\vee \neg \, x\in P$, but $x\notin P$, because otherwise $x/P=1/P\in {\cal B}(A)/P$. Since $P$ is a prime filter, it follows that $\neg \, x\in P$, that is $\neg \, x/P=1/P$, that is $\neg \, (x/P)=1/P$, thus $x/P=0/P\in {\cal B}(A)/P$, which is a contradiction to the choice of $x$. We have used Lemma \ref{aritmlr}, (\ref{aritmlr3}).

\noindent (\ref{psimaublp2}) By (\ref{psimaublp1}) and Lemma \ref{maxsiprime}, (\ref{maxsiprime1}).\end{proof}

Now let us focus on filters $F$ with ${\cal B}(p_F)$ injective.

\begin{proposition} For every filter $F$ of $A$, the following are equivalent:

\begin{enumerate}
\item\label{filtcuinj1} ${\cal B}(p_F)$ is injective; 
\item\label{filtcuinj2} ${\cal B}(A)\cap F=\{1\}$.
\end{enumerate}\label{filtcuinj}\end{proposition}

\begin{proof} (\ref{filtcuinj2})$\Rightarrow $(\ref{filtcuinj1}): Assume that ${\cal B}(A)\cap F=\{1\}$ and let $x,y\in {\cal B}(A)$ such that ${\cal B}(p_F)(x)={\cal B}(p_F)(y)$, that is $p_F(x)=p_F(y)$, which means that $x/F=y/F$, that is $x\leftrightarrow y\in F$. But then $x\leftrightarrow y\in {\cal B}(A)$ also, hence $x\leftrightarrow y\in {\cal B}(A)\cap F=\{1\}$, so $x\leftrightarrow y=1$, which means that $x=y$, according to Lemma \ref{aritmlr}, (\ref{aritmlr2}). 

\noindent (\ref{filtcuinj1})$\Rightarrow $(\ref{filtcuinj2}): Assume that ${\cal B}(p_F)$ is injective, and let $x\in {\cal B}(A)\cap F$, so $x\in F$, thus $x/F=1/F$, that is $p_F(x)=p_F(1)$, which means that ${\cal B}(p_F)(x)={\cal B}(p_F)(1)$ since both $x,1\in {\cal B}(A)$. The injectivity of ${\cal B}(p_F)$ now shows that $x=1$. Hence ${\cal B}(A)\cap F\subseteq \{1\}$. But $1\in {\cal B}(A)$ and $1\in F$, thus $\{1\}\subseteq {\cal B}(A)\cap F$. Therefore ${\cal B}(A)\cap F=\{1\}$.\end{proof}

\begin{corollary} If ${\cal B}(A)=\{0,1\}$, then every proper filter $F$ of $A$ has ${\cal B}(p_F)$ injective.\label{pfauinj}\end{corollary}

\begin{corollary} If $(F_i)_{i\in I}$ is a non--empty family of filters of $A$ such that ${\cal B}(p_{F_i})$ is injective for every $i\in I$, then ${\cal B}(p_{\bigcap _{i\in I}F_i})$ is injective.\end{corollary}

\begin{corollary}
\begin{enumerate}
\item\label{filtinrad1} Any filter $F$ of $A$ such that $F\subseteq {\rm Rad}(A)$ has ${\cal B}(p_F)$ injective.
\item\label{filtinrad2} ${\cal B}(p_{D(A)})$ is injective.
\end{enumerate}\label{filtinrad}\end{corollary}

\begin{proof} (\ref{filtinrad1}) Lemma \ref{lemaradsid}, (\ref{boolsirad}), and the fact that every filter contains 1 shows that each filter $F$ that is included in the radical of $A$ satisfies: ${\cal B}(A)\cap F=\{1\}$. Now apply Proposition \ref{filtcuinj}.

\noindent (\ref{filtinrad2}) By Proposition \ref{filtcuinj} and Lemma \ref{lemaradsid}, (\ref{consecinta}), or simply by (\ref{filtinrad1}) above and Lemma \ref{lemaradsid}, (\ref{dense}).\end{proof}

\section{Characterization of Boolean Lifting Property}
\label{charBLP}

In this section we prove two characterization theorems for residuated lattices with BLP (Proposition \ref{caractlrblp} and \ref{qlocblp}), which are then being used to obtain further properties and examples for residuated lattices with BLP.

Throughout this section, unless mentioned otherwise, $A$ will be an arbitrary residuated lattice.

\begin{lemma} For all $x\in A$, $x/[x\vee \neg \, x)\in {\cal B}(A/[x\vee \neg \, x))$.\label{olema}\end{lemma}

\begin{proof} Every $x\in A$ satisfies $x\vee \neg \, x\in [x\vee \neg \, x)$, hence $x/[x\vee \neg \, x)\in {\cal B}(A/[x\vee \neg \, x))$ by Lemma \ref{lifted}, (\ref{lifted3}).\end{proof}

\begin{lemma} For every $a,x\in A$:

\begin{enumerate}

\item\label{altalema0} for all $n\in \N ^*$, $x^n\leq a\rightarrow x$ and $(\neg \, x)^n\leq x\rightarrow a$;
\item\label{altalema1} there exists $k\in \N ^*$ such that $x^k\leq x\rightarrow a$ iff there exists $n\in \N ^*$ such that $x^n\leq a$;
\item\label{altalema2} there exists $k\in \N ^*$ such that $(\neg \, x)^k\leq a\rightarrow x$ iff there exists $n\in \N ^*$ such that $(\neg \, x)^n\leq \neg \, a$;
\item\label{altalema3} $a\leftrightarrow x\in [x\vee \neg \, x)$ iff $a\in [x)$ and $\neg \, a\in [\neg \, x)$.

\end{enumerate}\label{altalema}\end{lemma}

\begin{proof} Let $a,x\in A$.

\noindent (\ref{altalema0}) Let $n\in \N ^*$. By Lemma \ref{aritmlr}, (\ref{aritmlr1}) and (\ref{aritmlr8}), $x^n\leq x\leq a\rightarrow x$. Again by Lemma \ref{aritmlr}, (\ref{aritmlr1}) and (\ref{aritmlr6}), $(\neg \, x)^n\leq \neg \, x=x\rightarrow 0\leq x\rightarrow a$.

\noindent (\ref{altalema1}) If there exists $k\in \N ^*$ such that $x^k\leq x\rightarrow a$, then, by Lemma \ref{aritmlr}, (\ref{aritmlr1}) and (\ref{aritmlr4}), it follows that $x^{k+1}=x\odot x^k\leq x\odot (x\rightarrow a)\leq a$, hence $x^{k+1}\leq a$.

Conversely, if there exists $n\in \N ^*$ such that $x^n\leq a$, then, by Lemma \ref{aritmlr}, (\ref{aritmlr1}) and (\ref{aritmlr0}), it follows that $x^{n+1}=x\odot x^n\leq x\odot a\leq x\rightarrow a$, thus $x^{n+1}\leq x\rightarrow a$.

\noindent (\ref{altalema2}) If there exists $k\in \N ^*$ such that $(\neg \, x)^k\leq a\rightarrow x$, then, by Lemma \ref{aritmlr}, (\ref{aritmlr5}), and (\ref{altalema1}) in the present lemma, we get $(\neg \, x)^k\leq \neg \, x\rightarrow \neg \, a$, hence there exists $n\in \N ^*$ such that $(\neg \, x)^n\leq \neg \, a$.

Conversely, if there exists $n\in \N ^*$ such that $(\neg \, x)^n\leq \neg \, a$, then, by Lemma \ref{aritmlr}, (\ref{aritmlr1}) and (\ref{aritmlr7}), we obtain $(\neg \, x)^n\odot a=a\odot (\neg \, x)^n\leq a\odot \neg \, a=0\leq x$, so $(\neg \, x)^n\odot a\leq x$, which is equivalent to $(\neg \, x)^n\leq a\rightarrow x$ by the law of residuation.

\noindent (\ref{altalema3}) According to (\ref{altalema0}), (\ref{altalema1}) and (\ref{altalema2}) above, the following equivalences hold: $a\leftrightarrow x\in [x\vee \neg \, x)$ iff $a\leftrightarrow x\in [x)\cap [\neg \, x)$ iff $a\leftrightarrow x\in [x)$ and $a\leftrightarrow x\in [\neg \, x)$ iff there exist $k,j\in \N ^*$ such that $x^k\leq a\leftrightarrow x$ and $(\neg \, x)^j\leq a\leftrightarrow x$ iff there exist $k,j\in \N ^*$ such that $x^k\leq a\rightarrow x$, $x^k\leq x\rightarrow a$, $(\neg \, x)^j\leq a\rightarrow x$ and $(\neg \, x)^j\leq x\rightarrow a$ iff there exist $k,j\in \N ^*$ such that $x^k\leq x\rightarrow a$ and $(\neg \, x)^j\leq a\rightarrow x$ iff there exist $m,n\in \N ^*$ such that $x^m\leq a$ and $(\neg \, x)^n\leq \neg \, a$ iff $a\in [x)$ and $\neg \, a\in [\neg \, x)$.\end{proof}

\begin{lemma} For every $x\in A$, the following are equivalent:

\begin{itemize}
\item there exists $e\in {\cal B}(A)$ such that $e\leftrightarrow x\in [x\vee \neg \, x)$;
\item there exists $e\in {\cal B}(A)$ such that $e\in [x)$ and $\neg \, e\in [\neg \, x)$.
\end{itemize}\label{oaltalema}\end{lemma}

\begin{proof} By Lemma \ref{altalema}, (\ref{altalema3}).\end{proof}

\begin{notation} We shall denote $S(A)=\{x\in A\ |\ (\exists \, e\in {\cal B}(A))\, e\leftrightarrow x\in [x\vee \neg \, x)\}$.\end{notation}

\begin{remark} According to Lemma \ref{oaltalema}, $S(A)=\{x\in A\ |\ (\exists \, e\in {\cal B}(A))\, e\in [x)\mbox{ and }\neg \, e\in [\neg \, x)\}$.\label{rsa}\end{remark}

\begin{proposition} The following statements are equivalent:

\begin{enumerate}
\item\label{caractlrblp1} $A$ has BLP;
\item\label{caractlrblp2} for all $x\in A$, there exists $e\in {\cal B}(A)$ such that $e\leftrightarrow x\in [x\vee \neg \, x)$;
\item\label{caractlrblp3} for all $x\in A$, there exists $e\in {\cal B}(A)$ such that $e\in [x)$ and $\neg \, e\in [\neg \, x)$;
\item\label{caractlrblp4} $S(A)=A$.\end{enumerate}\label{caractlrblp}\end{proposition}

\begin{proof} (\ref{caractlrblp1})$\Rightarrow $(\ref{caractlrblp2}): Let $x\in A$. By the hypothesis of this implication, the filter $[x\vee \neg \, x)$ of $A$ has BLP, which means that ${\cal B}(A/[x\vee \neg \, x))={\cal B}(A)/[x\vee \neg \, x)$. By Lemma \ref{olema}, $x/[x\vee \neg \, x)\in {\cal B}(A/[x\vee \neg \, x))$, thus $x/[x\vee \neg \, x)\in {\cal B}(A)/[x\vee \neg \, x)$, so there exists $e\in {\cal B}(A)$ such that $x/[x\vee \neg \, x)=[x\vee \neg \, x)$, that is $e\leftrightarrow x\in [x\vee \neg \, x)$.

\noindent (\ref{caractlrblp2})$\Rightarrow $(\ref{caractlrblp1}): Let $F$ be an arbitrary filter of $A$ and $x\in A$, such that $x/F\in {\cal B}(A/F)$. Then, by Lemma \ref{lifted}, (\ref{lifted3}), $x\vee \neg \, x\in F$, thus $[x\vee \neg \, x)\subseteq F$. By the hypothesis of this implication, there exists $e\in {\cal B}(A)$ such that $e\leftrightarrow x\in [x\vee \neg \, x)\subseteq F$, hence $e\leftrightarrow x\in F$, that is $e/F=x/F$, therefore $x/F\in {\cal B}(A)/F$. So ${\cal B}(A/F)\subseteq {\cal B}(A)/F$, hence ${\cal B}(A/F)={\cal B}(A)/F$ by Lemma \ref{lifted}, (\ref{lifted4}), which means that $F$ has BLP. Thus $A$ has BLP.

\noindent (\ref{caractlrblp2})$\Leftrightarrow $(\ref{caractlrblp3}): By Lemma \ref{oaltalema}.

\noindent (\ref{caractlrblp2})$\Leftrightarrow $(\ref{caractlrblp4}): By the very definition of $S(A)$.\end{proof}

Condition (\ref{caractlrblp3}) in Proposition \ref{caractlrblp} is somewhat similar to condition $(4)$ from \cite[Proposition $1.1$]{nic}.

\begin{note} If we denote, for each $x\in A$, $S(x)=\{e\in {\cal B}(A)\ |\ e\in [x)\mbox{ and }\neg \, e\in [\neg \, x)\}$, then Proposition \ref{caractlrblp} shows that: $A$ has BLP iff $S(x)\neq \emptyset $ for all $x\in A$. Two open problems which are worth looking into are the following:

\begin{enumerate}
\item characterize the residuated lattices $A$ with $|S(x)|=1$ for all $x\in A$;
\item given an $n\in \N ^*$, characterize the residuated lattices $A$ with $|S(x)|\in \overline{1,n}$ for all $x\in A$.
\end{enumerate}\end{note}

According to the proposition above, a residuated lattice $A$ has BLP iff $S(A)=A$. Then, by Remarks \ref{boolblp} and \ref{chainblp}:

\begin{corollary}
\begin{enumerate}
\item\label{saea1} If (the bounded lattice reduct of) $A$ is a Boolean algebra, then $S(A)=A$.
\item\label{saea2} If $A$ is linearly ordered, then $S(A)=A$.
\end{enumerate}\label{saea}\end{corollary}

Let us give an example of a residuated lattice $A$ in which $S(A)\neq A$. In order to calculate easily the set $S(A)$ in this example, let us notice that:

\begin{remark} If ${\cal B}(A)=\{0,1\}$, then, according to Remark \ref{rsa} and Lemma \ref{aritmlr}, (\ref{aritmlr3}), $S(A)$ is formed of the elements $x\in A$ which satisfy one of the following conditions:

\begin{itemize}
\item $0\in [x)$ and $1\in [\neg \, x)$, that is $0\in [x)$, which means that $x^n=0$ for some $n\in \N ^*$;
\item $1\in [x)$ and $0\in [\neg \, x)$, that is $0\in [\neg \, x)$, which means that $(\neg \, x)^n=0$ for some $n\in \N ^*$.
\end{itemize}

Thus, when ${\cal B}(A)=\{0,1\}$, it follows that $S(A)$ contains exactly the elements $x\in A$ such that either $x$ or $\neg \, x$ is nilpotent, that is: $S(A)=N(A)\cup \{x\in A\ |\ \neg \, x\in N(A)\}$.\label{booltriv}\end{remark}

\begin{remark}
If ${\cal B}(A)=\{0,1\}$ and $\odot =\wedge $ in $A$, then all the elements of $A$ are idempotent, hence $N(A)=\{0\}$, thus, by Remark \ref{booltriv}, $S(A)=\{0\}\cup \{x\in A\ |\ \neg \, x=0\}=\{0\}\cup D(A)$.\label{fpartic}\end{remark}

\begin{example} Let $A$ be the residuated lattice in Example \ref{exfarablp}. In this residuated lattice, ${\cal B}(A)=\{0,1\}$ and $\odot =\wedge $, hence, by Remark \ref{fpartic}, $S(A)=\{0\}\cup D(A)=\{0\}\cup \{x\in A\ |\ \neg \, x=0\}=\{0\}\cup \{x\in A\ |\ x\rightarrow 0=0\}=\{0,c,1\}\subsetneq A$.\label{exsanua}\end{example}

\begin{remark} If ${\rm Rad}(A)=A\setminus \{0\}$, then $A$ is local. This is obvious, because the previous equality says that $0\notin {\rm Rad}(A)$ and ${\rm Rad}(A)\cup \{0\}=A$, and, since ${\rm Rad}(A)$ is a filter of $A$, this means that ${\rm Rad}(A)$ is a proper filter and the only filter that strictly includes it is $A$, therefore ${\rm Rad}(A)$ is a maximal filter of $A$, hence $A$ is local.

This actually also follows from the equivalence between (\ref{caractloc1}) and (\ref{caractloc2}) in Proposition \ref{caractloc}, which can be given a proof similar to the above.\label{evident}\end{remark}

\begin{remark} If $A$ is non—trivial, then, by Lemma \ref{aritmlr}, (\ref{aritmlr3}), $\neg \, 0=1\neq 0$, thus $0\notin D(A)$.\label{feasy}\end{remark}

\begin{example} If $A$ is a non--trivial chain, organized as a residuated lattice as in Example \ref{exutil}, then:

\begin{itemize}
\item by Remark \ref{feasy}, $0\notin D(A)$;

\item $S(A)=A$, by Corollary \ref{saea}, (\ref{saea2});
\item ${\cal B}(A)=\{0,1\}$ and $\odot =\wedge $ in $A$, hence $S(A)=\{0\}\cup D(A)$ by Remark \ref{fpartic}.
\end{itemize}

Therefore $D(A)=A\setminus \{0\}$. Also, by Lemma \ref{lemaradsid}, (\ref{dense}), and the fact that ${\rm Rad}(A)$ is a proper filter of $A$, it follows that ${\rm Rad}(A)=A\setminus \{0\}$, thus $A$ is local by Remark \ref{evident}.\label{easy}\end{example}

\begin{proposition}
\begin{enumerate}
\item\label{colectare1} ${\cal B}(A)\subseteq S(A)$.

\item\label{colectare2} If $x\in A$ such that $\neg \, x\in S(A)$, then $x\in S(A)$.

\item\label{colectare3} If $x\in A$ such that $x^n\in S(A)$ for some $n\in \N ^*$, then $x\in S(A)$.

\item\label{colectare4} $D(A)\subseteq S(A)$.

\item\label{colectare5} ${\rm Rad}(A)\subseteq S(A)$.

\item\label{colectare6} For any filter $F$ of $A$, $S(A)/F\subseteq S(A/F)$.\end{enumerate}\label{colectare}\end{proposition}

\begin{proof} (\ref{colectare1}) For all $e\in {\cal B}(A)$, by Lemma \ref{aritmlr}, (\ref{aritmlr2}), we have: $e\leftrightarrow e=1\in [e\vee \neg \, e)$, therefore $e\in S(A)$.

\noindent (\ref{colectare2}) Let $x\in A$ such that $\neg \, x\in S(A)$. By Lemma \ref{aritmlr}, (\ref{aritmlr10}), $x\leq \neg \, \neg \, x$, thus $[\neg \, \neg \, x)\subseteq [x)$. According to Remark \ref{rsa}, $\neg \, x\in S(A)$ means that there exists $e\in {\cal B}(A)$ such that $e\in [\neg \, x)$ and $\neg \, e\in [\neg \, \neg \, x)\subseteq [x)$, thus $\neg \, e\in [x)$ and $e=\neg \, \neg \, e\in [\neg \, x)$, and $\neg \, e\in {\cal B}(A)$, thus $x\in S(A)$. We have applied Lemma \ref{aritmcbool}, (\ref{aritmcbool1}).

\noindent (\ref{colectare3}) Let $x\in A$ and $n\in \N ^*$, such that $x^n\in S(A)$. By Lemma \ref{uncalcul}, $(\neg \, x)^n\leq \neg \, x^n$, hence $[\neg \, x^n)\subseteq [(\neg \, x)^n)=[\neg \, x)$. Remark \ref{rsa} shows that there exists $e\in {\cal B}(A)$ such that $e\in [x^n)=[x)$ and $\neg \, e\in [\neg \, x^n)\subseteq [\neg \, x)$, thus $e\in [x)$ and $\neg \, e\in [\neg \, x)$, that is $x\in S(A)$.

\noindent (\ref{colectare4}) Let $x\in D(A)$, that is $x\in A$ with $\neg \, x=0$. By (\ref{colectare1}) and (\ref{colectare2}), $0\in {\cal B}(A)$, hence $0\in S(A)$, therefore $x\in S(A)$. This, actually, also follows from the next statement, (\ref{colectare5}), and Lemma \ref{lemaradsid}, (\ref{dense}).

\noindent (\ref{colectare5}) Let $x\in {\rm Rad}(A)$, and let us take $n=1$ in Lemma \ref{radunit}. We get that there exists $k\in \N ^*$ such that $(\neg \, x)^k=0$. But $0\in {\cal B}(A)\subseteq S(A)$ by (\ref{colectare1}), thus $(\neg \, x)^k\in S(A)$, hence $\neg \, x\in S(A)$, therefore $x\in S(A)$, as shown by (\ref{colectare3}) and (\ref{colectare2}).

\noindent (\ref{colectare6}) Let $F$ be a filter of $A$, and let us consider an arbitrary element of $S(A)/F$, that is let $x\in S(A)$ and let us consider the element $x/F\in S(A)/F$. By Remark \ref{rsa}, this means that $e\in [x)$ and $\neg \, e\in [\neg \, x)$ for some $e\in {\cal B}(A)$, so there exist $n,m\in \N ^*$ such that $x^n\leq e$ and $(\neg \, x)^m\leq \neg \, e$. Then $e/F\in {\cal B}(A)/F\subseteq {\cal B}(A/F)$, $(x/F)^n=x^n/F\leq e/F$ and $\neg \, (x/F)^m=(\neg \, x)^m/F\leq \neg \, e/F=\neg \, (e/F)$, therefore $x/F\in S(A/F)$, by Lemma \ref{lifted}, (\ref{lifted4}) and again Remark \ref{rsa}.\end{proof}

\begin{corollary} $A$ has BLP iff $A/F$ has BLP for every filter $F$ of $A$.\label{corimp}\end{corollary}

\begin{proof} For the direct implication, notice that: by Proposition \ref{caractlrblp} and Proposition \ref{colectare}, (\ref{colectare6}), if $A$ has BLP, then $S(A)=A$, hence $A/F=S(A)/F\subseteq S(A/F)\subseteq A/F$, hence $S(A/F)=A/F$, thus $A/F$ has BLP. For the converse implication, just take $F=\{1\}$.\end{proof}

\begin{proposition}
For every filter $F$ of $A$, the following are equivalent:

\begin{enumerate}
\item\label{catdecaturi1} $A/F$ has BLP;
\item\label{catdecaturi2} for every filter $G$ of $A$ such that $F\subseteq G$, $A/G$ has BLP.\end{enumerate}\label{catdecaturi}\end{proposition}

\begin{proof} (\ref{catdecaturi2})$\Rightarrow $(\ref{catdecaturi1}): Take $F=G$ in (\ref{catdecaturi2}).

\noindent (\ref{catdecaturi1})$\Rightarrow $(\ref{catdecaturi2}): Let $G$ be a filter of $A$ such that $F\subseteq G$. By the hypothesis of this implication, $A/F$ has BLP. We know that $G/F$ is a filter of $A/F$. Thus, by Corollary \ref{corimp}, it follows that $(A/F)/_{\textstyle (G/F)}$ has BLP. But the residuated lattice $(A/F)/_{\textstyle (G/F)}$ is isomorphic to $A/G$. Hence $A/G$ has BLP.\end{proof} 

\begin{corollary}
Any hyperarchimedean residuated lattice has BLP, but the converse is not true.\label{hypblp}\end{corollary}

\begin{proof} Let $A$ be a hyperarchimedean residuated lattice and $x\in A$, arbitrary. Then there exists an $n\in \N ^*$ such that $x^n\in {\cal B}(A)$, so $x^n\in S(A)$, hence $x\in S(A)$, by Proposition \ref{colectare}, (\ref{colectare1}) and (\ref{colectare3}). Thus $S(A)=A$, which means that $A$ has BLP, according to Proposition \ref{caractlrblp}.

Now let $A$ be a chain with at least three elements, organized as a residuated lattice as in Example \ref{exutil}. Then $A$ has BLP by Remark \ref{chainblp}, ${\cal B}(A)=\{0,1\}\neq A$ and $\odot =\wedge $, thus all elements of $A$ are idempotent, hence none of the elements of $A\setminus {\cal B}(A)=A\setminus \{0,1\}\neq \emptyset $ is archimedean, therefore $A$ is not hyperarchimedean.\end{proof}

\begin{remark}
Let us point out the property which we have used at the end of the proof of the previous corollary: if $\odot =\wedge $ in the residuated lattice $A$, then all elements of $A$ are idempotent, hence the set of the archimedean elements of $A$ coincides with ${\cal B}(A)$, therefore: $A$ is hyperarchimedean iff ${\cal B}(A)=A$ iff (the bounded lattice reduct of) $A$ is a Boolean algebra. Thus: hyperarchimedean residuated lattices with $\odot =\wedge $ are Boolean algebras.\label{aplicabila}\end{remark}

\begin{remark}
Remark \ref{aplicabila} could have been used to disprove the converse implication in the previous corollary, by simply applying Remark \ref{chainblp} and the fact that the only linearly ordered Boolean algebras are the one--element chain and the two--element chain, in order to conclude that every bounded chain with at least three elements, organized as a residuated lattice as in Example \ref{exutil}, has BLP and is not hyperarchimedean.
 
Moreover, Remark \ref{aplicabila} shows that, if a residuated lattice with BLP has $\odot =\wedge $ and it is not a Boolean algebra, then it is not hyperarchimedean.\end{remark}

\begin{proposition}
If $\odot =\wedge $ in $A$, then $S(A)=\{x\in A\ |\ \neg \, x\in {\cal B}(A)\}$.\label{sinallid}\end{proposition}

\begin{proof} Assume that $A$ has $\odot =\wedge $. Then all the elements of $A$ are idempotent, so, for every $a\in A$, $[a)=\{b\in A\ |\ a\leq b\}$. Then, according to Remark \ref{rsa}, $S(A)=\{x\in A\ |\ (\exists \, e\in {\cal B}(A))\, e\in [x)\mbox{ and }\neg \, e\in [\neg \, x)\}=\{x\in A\ |\ (\exists \, e\in {\cal B}(A))\, x\leq e\mbox{ and }\neg \, x\leq \neg \, e\}$. But $x\leq e$ implies  $\neg \, e\leq \neg \, x$ by Lemma \ref{aritmlr}, (\ref{aritmlr6}), which means that $x\leq e$ and $\neg \, x\leq \neg \, e$ imply $\neg \, x=\neg \, e\in {\cal B}(A)$ (see Lemma \ref{aritmcbool}, (\ref{aritmcbool1})). By Lemma \ref{aritmcbool}, (\ref{aritmcbool1}) and Lemma \ref{aritmlr}, (\ref{aritmlr10}), if an $x\in A$ has $\neg \, x\in {\cal B}(A)$, then $\neg \, \neg \, x\in {\cal B}(A)$ and we have: $\neg \, \neg \, x\in [x)$ and $\neg \, \neg \, \neg \, x\in [\neg \, x)$. Therefore $S(A)=\{x\in A\ |\ \neg \, x\in {\cal B}(A)\}$.\end{proof}

\begin{corollary}

If $\odot =\wedge $ in $A$, then $\{\neg \, x\ |\ x\in S(A)\}={\cal B}(A)$.\label{negsb}\end{corollary}

\begin{proof} Lemma \ref{aritmcbool}, (\ref{aritmcbool1}), and Proposition \ref{sinallid} ensure us that, when $\odot =\wedge $ in $A$, we have: for every $x\in S(A)$, $\neg \, x\in {\cal B}(A)$, and, for every $y\in {\cal B}(A)$, $y=\neg \, \neg \, y=\neg \, x$, where $x=\neg \, y\in S(A)$.\end{proof}

\begin{corollary}
If $A$ is involutive and has $\odot =\wedge $, then $S(A)={\cal B}(A)$. In other words, if all elements of $A$ are both regular and idempotent, then $S(A)={\cal B}(A)$.\label{lovedata}\end{corollary}

\begin{proof} By Proposition \ref{colectare}, (\ref{colectare1}), ${\cal B}(A)\subseteq S(A)$. Now assume that $A$ is involutive and has $\odot =\wedge $, and let $x\in S(A)$. Then $\neg \, x\in {\cal B}(A)$ by Proposition \ref{sinallid}, thus $x=\neg \, \neg \, x\in {\cal B}(A)$ by Lemma \ref{aritmcbool}, (\ref{aritmcbool1}). Hence $S(A)\subseteq {\cal B}(A)$, thus $S(A)={\cal B}(A)$.\end{proof}

\begin{corollary}
If $A$ is involutive and has $\odot =\wedge $, then: $A$ has BLP iff (the underlying bounded lattice of) $A$ is a Boolean algebra iff $A$ is hyperarchimedean.\label{liketroy}\end{corollary}

\begin{proof} Proposition \ref{caractlrblp} and Corollary \ref{lovedata} show the first equivalence, and Remark \ref{aplicabila} proves the second equivalence.\end{proof}

\begin{remark} $0\in D(A)$ iff $A$ is trivial (see Remark \ref{feasy}), and $0\in {\cal B}(A)\subseteq S(A)$ (see Proposition \ref{colectare}, (\ref{colectare1})), therefore: $D(A)={\cal B}(A)$ iff $D(A)=S(A)$ iff $A$ is trivial.\end{remark}

\begin{proposition}
If $D(A)\subseteq {\cal B}(A)$, then $D(A)=\{1\}$. Consequently, if $D(A)\cup \{0\}={\cal B}(A)$, then $D(A)=\{1\}$ and ${\cal B}(A)=\{0,1\}$.\label{dsbools}\end{proposition}

\begin{proof} As shown by Lemma \ref{lemaradsid}, (\ref{consecinta}), ${\cal B}(A)\cap D(A)=\{1\}$, hence, if $D(A)\subseteq {\cal B}(A)$, then, $\{1\}={\cal B}(A)\cap D(A)=D(A)$. If, moreover, $D(A)\cup \{0\}={\cal B}(A)$, then $D(A)\subseteq {\cal B}(A)$, so $D(A)=\{1\}$, thus ${\cal B}(A)=\{0,1\}$.\end{proof}

\begin{proposition} If ${\cal B}(A)=S(A)$, then $D(A)={\rm Rad}(A)=\{1\}$.\label{cool}\end{proposition}

\begin{proof} Lemma \ref{lemaradsid}, (\ref{boolsirad}), and Proposition \ref{colectare}, (\ref{colectare5}), show that, if ${\cal B}(A)=S(A)$, then $\{1\}={\cal B}(A)\cap {\rm Rad}(A)=S(A)\cap {\rm Rad}(A)={\rm Rad}(A)$, hence ${\rm Rad}(A)=\{1\}$. Lemma \ref{lemaradsid}, (\ref{dense}), and the fact that every filter contains $1$ now imply $D(A)=\{1\}$.\end{proof}

\begin{corollary} If (the underlying bounded lattice of) $A$ is a Boolean algebra, then $D(A)={\rm Rad}(A)=\{1\}$.\end{corollary}

\begin{proof} If $A$ is a Boolean algebra, then ${\cal B}(A)=A=S(A)$, according to Corollary \ref{saea}, (\ref{saea1}), hence $D(A)={\rm Rad}(A)=\{1\}$ by Proposition \ref{cool}. Actually, this corollary can very easily be deduced from the definition of $D(A)$, Lemma \ref{radunit} and the fact that every element of $A$ is idempotent if $A$ is a Boolean algebra, because then $\odot =\wedge $ in $A$. Notice that, since $\odot =\wedge $ in this case, it follows that the residuated lattice filters of $A$ coincide with its Boolean filters, hence ${\rm Rad}(A)$ coincides with its radical as a Boolean algebra, making this an interesting approach to the same result on Boolean algebras.\end{proof}

\begin{corollary}

If $A$ is involutive and has $\odot =\wedge $, then $D(A)={\rm Rad}(A)=\{1\}$.\end{corollary}

\begin{proof} By Corollary \ref{lovedata} and Proposition \ref{cool}. Actually, this also follows immediately from the definition of $D(A)$ and Lemma \ref{radunit}.\end{proof}

Many results resembling the ones in the previous corollaries can be proved here, results which also have straightforward algebraic or arithmetic proofs; but it is interesting to see how arguments as the ones above ``close the circles`` in the study of these algebras, and this is also an illustration of the effectiveness of the mathematical techniques that derive from the study of BLP.

$A$ is said to be {\em quasi--local} iff, for all $a\in A$, there exist $e\in {\cal B}(A)$ and $n\in \N ^*$ such that $a^n\odot e=0$ and $(\neg \, a)^n\odot \neg \, e=0$ (\cite{eu4}). It is straightforward, from Lemma \ref{implicdir}, that any local residuated lattice is quasi--local (see also \cite{eu4}).

\begin{remark} If $(A_i)_{i\in I}$ is a non--empty family of residuated lattices and $\displaystyle A=\prod _{i\in I}A_i$, then:

\begin{enumerate}
\item if $A$ is quasi--local, then, for all $i\in I$, $A_i$ is quasi--local;
\item if either $I$ is finite or $\odot =\wedge $ (which means that all elements are idempotent) in these residuated lattices, then: $A$ is quasi--local iff, for all $i\in I$, $A_i$ is quasi--local.
\end{enumerate}

This is straightforward, by the very definition of quasi--local residuated lattices and the fact that $\displaystyle {\cal B}(A)=\prod _{i\in I}{\cal B}(A_i)$, but it also follows from Propositions \ref{qlocblp}, \ref{prodblp}, \ref{idempprodblp} and \ref{prodblprez} below.\end{remark}

A bounded distributive lattice $L$ is called a {\em B--normal lattice} iff, for all $x,y\in L$, if $x\vee y=1$, then there exist $e,f\in {\cal B}(L)$ such that $e\wedge f=0$ and $x\vee e=y\vee f=1$. B--normal lattices (actually their duals) have been studied it \cite{cig}.

\begin{proposition} The following are equivalent:

\begin{enumerate}
\item\label{qlocblp1} $A$ is quasi--local;
\item\label{qlocblp2} $A$ has BLP;
\item\label{qlocblp3} for all $x,y\in A$, if $[x)\vee [y)=A$, then there exist $e,f\in {\cal B}(A)$ such that $e\vee f=1$ and $[x)\vee [e)=[y)\vee [f)=A$;
\item\label{qlocblp4} the bounded distributive lattice ${\cal PF}(A)$ is dually B--normal;
\item\label{qlocblp5} the bounded distributive lattice $(A,\vee ,\odot ,0,1)$ is B--normal.\end{enumerate}\label{qlocblp}\end{proposition}

\begin{proof} (\ref{qlocblp1})$\Leftrightarrow $(\ref{qlocblp2}): $A$ is quasi--local iff, for all $a\in A$, there exist $e\in {\cal B}(A)$ and $n\in \N ^*$ such that $a^n\odot e=0$ and $(\neg \, a)^n\odot \neg \, e=0$, that is $a^n\leq \neg \, e$ and $(\neg \, a)^n\leq \neg \, \neg \, e$, according to Lemma \ref{aritmlr}, (\ref{aritmlr7}).

By Proposition \ref{caractlrblp}, $A$ has BLP iff, for all $a\in A$, there exists $f\in {\cal B}(A)$ such that $f\in [x)$ and $\neg \, f\in [\neg \, x)$, that is $a^m\leq f$ and $(\neg \, a)^k\leq \neg \, f$ for some $m,k\in \N ^*$.

Now the needed equivalence is clear: for the direct implication, take $f=\neg \, e$ (see Lemma \ref{aritmcbool}, (\ref{aritmcbool1})) and $m=k=n$; for the converse implication, take $e=\neg \, f$ (see again Lemma \ref{aritmcbool}, (\ref{aritmcbool1})) and $n=\max \{m,k\}$, and apply Lemma \ref{aritmlr}, (\ref{aritmlr1}).

\noindent (\ref{qlocblp3})$\Rightarrow $(\ref{qlocblp1}): Let $x\in A$, arbitrary. By Lemma \ref{aritmlr}, (\ref{aritmlr7}), $x\odot \neg \, x=0$, hence $[x)\vee [\neg \, x)=[x\odot \neg \, x)=A$. Now the hypothesis of this implication shows that there exist $e,f\in {\cal B}(A)$ such that $e\vee f=1$ and $[x)\vee [e)=[\neg \, x)\vee [f)=A$. From $e\vee f=1$, Lemma \ref{aritmcbool}, (\ref{aritmcbool4}), and Lemma \ref{aritmlr}, (\ref{aritmlr2}), give us $\neg \, e\rightarrow f=1$, that is $\neg \, e\leq f$. $[0)=A=[x)\vee [e)=[x\odot e)$ means that $x^n\odot e=x^n\odot e^n=(x\odot e)^n=0$ for some $n\in \N ^*$, so $x^n\odot e=0$, thus $x^n\leq \neg \, e\leq f$, hence $\neg \, e\in [x)$ and $f\in [x)$; we have applied Lemma \ref{aritmcbool}, (\ref{aritmcbool2}) and Lemma \ref{aritmlr}, (\ref{aritmlr7}). Analogously, $A=[\neg \, x)\vee [f)$ implies $\neg \, f\in [\neg \, x)$. So $f\in {\cal B}(A)$, $f\in [x)$ and $\neg \, f\in [\neg \, x)$. Therefore $A$ has BLP, according to Proposition \ref{caractlrblp}.

\noindent (\ref{qlocblp1})$\Rightarrow $(\ref{qlocblp3}): Let $x,y\in A$, such that $[x\odot y)=[x)\vee [y)=A=[0)$, which means that $x^n\odot y^n=(x\odot y)^n=0$ for some $n\in \N ^*$. The hypothesis of this implication states that $A$ has BLP. Then, according to Proposition \ref{caractlrblp}, there exist $e,f\in {\cal B}(A)$ such that $e\in [x^n)$, $\neg \, e\in [\neg \, x^n)$, $f\in [y^n)$ and $\neg \, f\in [\neg \, y^n)$, thus $x^{np}\leq e$, $(\neg \, x^n)^q\leq \neg \, e$, $y^{nr}\leq f$ and $(\neg \, y^n)^s\leq \neg \, f$ for some $p,q,r,s\in \N ^*$. Let $k=\max \{p,q,r,s\}\in \N ^*$. By Lemma \ref{aritmlr}, (\ref{aritmlr1}), it follows that $x^{nk}\leq e$, $(\neg \, x^n)^k\leq \neg \, e$, $y^{nk}\leq f$ and $(\neg \, y^n)^k\leq \neg \, f$. By Lemma \ref{aritmlr}, (\ref{aritmlr7}) and (\ref{aritmlr1}), from $x^n\odot y^n=0$ we obtain: $x^n\leq \neg \, y^n$, hence $x^{nk}\leq (\neg \, y^n)^k\leq \neg \, f$, and $y^n\leq \neg \, x^n$, hence $y^{nk}\leq (\neg \, x^n)^k\leq \neg \, e$. So, by Lemma \ref{aritmcbool}, (\ref{aritmcbool2}), $x^{nk}\leq e$ and $x^{nk}\leq \neg \, f$, thus $x^{nk}\leq \neg \, f\wedge e=\neg \, f\odot e$, while $y^{nk}\leq f$ and $y^{nk}\leq \neg \, e$, thus $y^{nk}\leq \neg \, e\wedge f=\neg \, e\odot f$. We denote $a=\neg \, f\odot e\in {\cal B}(A)$ and $b=\neg \, e\odot f\in {\cal B}(A)$, hence $\neg \, a,\neg \, b\in {\cal B}(A)$. By Lemma \ref{aritmlr}, (\ref{aritmlr7}), and Lemma \ref{aritmcbool}, (\ref{aritmcbool1}), $x^{nk}\leq a=\neg \, \neg \, a$ and $y^{nk}\leq b=\neg \, \neg \, b$ are equivalent to $x^{nk}\odot \neg \, a=0$ and $y^{nk}\odot \neg \, b=0$, from which we get $A=[0)=[x^{nk}\odot \neg \, a)=[x^{nk})\vee [\neg \, a)=[x)\vee [\neg \, a)$ and $A=[0)=[y^{nk}\odot \neg \, b)=[y^{nk})\vee [\neg \, b)=[y)\vee [\neg \, b)$. By Lemma \ref{aritmcbool}, (\ref{aritmcbool1}) and (\ref{aritmcbool2}), and Lemma \ref{aritmlr}, (\ref{aritmlr7}), $a\odot b=\neg \, e\odot f\odot \neg \, f\odot e=e\odot \neg \, e\odot f\odot \neg \, f=0\odot 0=0$, thus $1=\neg \, 0=\neg \, (a\odot b)=\neg \, (a\wedge b)=\neg \, a\vee \neg \, b$.

\noindent (\ref{qlocblp4})$\Leftrightarrow $(\ref{qlocblp5}): By the fact that the bounded lattice ${\cal PF}(A)$ is isomorphic to the dual of $(A,\vee ,\odot ,0,1)$.

\noindent (\ref{qlocblp3})$\Leftrightarrow $(\ref{qlocblp4}): (\ref{qlocblp3}) states exactly the fact that the bounded distributive lattice ${\cal PF}(A)$ is dually B--normal, because: $e\vee f=1$ is equivalent to $\{1\}=[1)=[e\vee f)=[e)\cap [f)$, and we have: $\lambda :A\rightarrow {\cal PF}(A)$, defined by $\lambda (a)=[a)$ for all $a\in A$, is a bounded lattice isomorphism between $(A,\vee ,\odot ,0,1)$ and the dual of {\cal PF}(A), the notion of Boolean center is clearly dual to itself, and the Boolean center ${\cal B}(A)$ of the residuated lattice $A$ coincides with the Boolean center of the bounded distributive lattice $(A,\vee ,\odot ,0,1)$, hence ${\cal B}({\cal PF}(A))=\lambda ({\cal B}(A))=\{\lambda (e)\ |\ e\in {\cal B}(A)\}=\{[e)\ |\ e\in {\cal B}(A)\}$.\end{proof}

\begin{note} The equivalences in the previous proposition set residuated lattices with BLP in relation to two important classes of algebras: quasi--local residuated lattices and B--normal lattices. These connections will allow a transfer of properties to be made between these three types of structures.\end{note}

\begin{note} In \cite{figg}, quasi--local MV--algebras have been introduced as a generalization of local MV--algebras, and they have been characterized as weak--Boolean products of local MV--algebras. This result has been extended to quasi--local BL--algebras in \cite{adinggll} (see also \cite{leo}). These works contain many other characterizations for quasi--local MV--algebras and BL--algebras. In \cite{eu4}, quasi--local residuated lattices have been studied in relation to certain types of filters of a residuated lattice, the set of the dense elements and Glivenko residuated lattices. It remains an open problem to determine whether quasi--local residuated lattices (and thus residuated lattices with BLP) can be represented as weak--Boolean products of local residuated lattices.\end{note}

\begin{corollary}
Any local residuated lattice has BLP. Moreover, if $A$ is a local residuated lattice and $F$ is a proper filter of $A$, then ${\cal B}(p_F)$ is a bijection.\label{locblp}\end{corollary}

\begin{proof} Let $A$ be a local residuated lattice. Then $A$ is quasi--local, hence $A$ has BLP by Proposition \ref{qlocblp}.

The second statement in the enunciation can be proved in more than one way at this point. Certainly, the surjectivity of ${\cal B}(p_F)$ comes from the BLP; its injectivity can be shown by Corollary \ref{filtinrad}, (\ref{filtinrad1}), Lemma \ref{maxsiprime}, (\ref{maxsiprime2}), and the fact that the only maximal filter of $A$ is ${\rm Rad}(A)$, or directly from Corollary \ref{pfauinj} and the fact that ${\cal B}(A)=\{0,1\}$.\end{proof}

In the following two remarks, we shall use Remark \ref{interestingfact} and the fact that the filters of a residuated lattice with $\odot =\wedge $ coincide with the filters of its bounded lattice reduct.

\begin{remark} There exist quasi--local residuated lattices which are not local. By Proposition \ref{qlocblp}, these are exactly those residuated lattices that have BLP, but are not local, which proves that the converse of the first statement in Corollary \ref{locblp} is not true.

Here is an example of a quasi--local residuated lattice which is not local: let $A$ be the direct product between the two--element chain and itself, which is a Boolean algebra, hence a residuated lattice with BLP, according to Remark \ref{boolblp}. Then $A$ is not local, because it has two distinct maximal filters. In general, for any $n\in \N ^*$, the Boolean algebra induced by the $n$--th power of the two--element chain has BLP, according to Remark \ref{boolblp}, thus it is quasi--local, by Proposition \ref{qlocblp}, but it has exactly $n$ different maximal filters, so, for $n\geq 2$, it is not local.\label{qlocnotloc}\end{remark}

\begin{remark} There exist quasi--local residuated lattices which are not semilocal. By Proposition \ref{qlocblp}, this means that there exist residuated lattices with BLP which are not semilocal. In order to find one such example, we shall adapt the example in the previous remark, by turning the nonzero natural exponent to an infinite set. Let us denote by ${\cal L}_2$ the two--element chain, and let $I$ be an infinite set. Then ${\cal L}_2^I$ is a Boolean algebra, thus a residuated lattice with BLP, by Remark \ref{boolblp}, hence a quasi--local residuated lattice, according to Proposition \ref{qlocblp}. But ${\rm Max}({\cal L}_2^I)$ is in bijection to $I$, thus it is infinite, so ${\cal L}_2^I$ is not a semilocal residuated lattice.\label{qlocnotsloc}\end{remark}

\section{Boolean Lifting Property and Direct Products of Residuated Lattices}
\label{prods}

The purpose of this section is to study the behaviour of BLP with respect to direct products of residuated lattices. We shall prove that a finite direct product of residuated lattices has BLP iff each residuated lattice in the product has BLP and, moreover, this holds for individual filters. Weaker results hold for arbitrary direct products of residuated lattices.

Until mentioned otherwise, let $(A_i)_{i\in I}$ be a non--empty family of residuated lattices and $\displaystyle A=\prod _{i\in I}A_i$. Since residuated lattices form an equational class, it follows that $A$ becomes a residuated lattice with the operations defined canonically, that is componentwise. Also, clearly, $\odot =\wedge $ in $A$ iff $\odot =\wedge $ in $A_i$, for each $i\in I$.

It is immediate that $\displaystyle {\cal B}(A)=\prod _{i\in I}{\cal B}(A_i)$.

Now let us denote, for each $i\in I$, the canonical projection $pr_i:A\rightarrow A_i$, which is, obviously, a surjective residuated lattice morphism, hence, for every filter $F$ of $A$, $pr_i(F)$ is a filter of $A_i$. Also, if, for each $i\in I$, $F_i$ is a filter of $A_i$, then clearly $\displaystyle F=\prod _{i\in I}F_i$ is a filter of $A$ with $pr_i(F)=F_i$ for all $i\in I$. This means that the mapping $\displaystyle F\rightarrow (pr_i(F))_{i\in I}$ from ${\cal F}(A)$ to $\displaystyle \prod _{i\in I}{\cal F}(A_i)$ is well defined and surjective. Clearly every filter $F$ of $A$ satisfies $\displaystyle F\subseteq \prod _{i\in I}pr_i(F)$, and it is straightforward that, when $I$ is finite, say $I=\overline{1,n}$, with $n\in \N ^*$, then the converse inclusion holds as well, that is every filter $F$ of $A$ has the property that $\displaystyle F=\prod _{i=1}^npr_i(F)$, which shows that, in this case, the mapping above is also injective, thus it is a bijection between ${\cal F}(A)$ and $\displaystyle \prod _{i=1}^n{\cal F}(A_i)$. 

It is straightforward, from Lemma \ref{radunit}, that $\displaystyle {\rm Rad}(A)\subseteq \prod _{i\in I}{\rm Rad}(A_i)$ and, if $I$ is finite, say $I=\overline{1,n}$, with $n\in \N ^*$, then the converse inclusion holds as well, so that $\displaystyle {\rm Rad}(A)=\prod _{i=1}^n{\rm Rad}(A_i)$.

Also, according to \cite[Theorem $3.3$]{eu3}, if $I$ is finite, $I=\overline{1,n}$ with $n\in \N ^*$, then $\displaystyle |{\rm Max}(A)|=\sum _{i=1}^n|{\rm Max}(A_i)|$.

The results that follow refer to products of non--empty families of residuated lattices. We have separated the results that only hold for finite non--empty products from those which hold in the general case. We have also pointed out the results which hold for arbitrary non--empty families of residuated lattices all having $\odot =\wedge $, that is all having all elements idempotent; of course, the direct product of such a family also has $\odot =\wedge $, that is it has all elements idempotent. We will do the same for all following results involving direct products.

\begin{lemma}
Let $(A_i)_{i\in I}$ be a non--empty family of residuated lattices, and $\displaystyle A=\prod _{i\in I}A_i$. Then $\displaystyle S(A)=S(\prod _{i\in I}A_i)\subseteq \prod _{i\in I}S(A_i)$.\label{proddirl0}\end{lemma}

\begin{proof} We shall be using the fact that $\displaystyle {\cal B}(A)={\cal B}(\prod _{i\in I}A_i)=\prod _{i\in I}{\cal B}(A_i)$, along with Remark \ref{rsa}. For every $x\in S(A)$, there exists $e\in {\cal B}(A)$ such that $e\in [x)$ and $\neg \, e\in [\neg \, x)$. But then $x=(x_i)_{i\in I}$ and $e=(e_i)_{i\in I}$, with $x_i\in A_i$ and $e_i\in {\cal B}(A_i)$ for each $i\in I$, and it is immediate that, for all $i\in I$, we have $e_i\in [x_i)$ and $\neg \, e_i\in [\neg \, x_i)$, which means that $x_i\in S(A_i)$. Thus $\displaystyle S(A)=S(\prod _{i\in I}A_i)\subseteq \prod _{i\in I}S(A_i)$.\end{proof}

\begin{lemma}
Let $n\in \N ^*$, $(A_i)_{i=1}^n$ be a family of residuated lattices, and $\displaystyle A=\prod _{i=1}^nA_i$. Then $\displaystyle S(A)=S(\prod _{i=1}^nA_i)=\prod _{i=1}^nS(A_i)$.\label{proddirl1}\end{lemma}

\begin{proof} The fact that $\displaystyle S(A)=S(\prod _{i=1}^nA_i)\subseteq \prod _{i=1}^nS(A_i)$ follows from Lemma \ref{proddirl0}. Now let $x=(x_1,\ldots ,x_n)\in \prod _{i=1}^nS(A_i)$, that is $x_i\in S(A_i)$ for each $i\in \overline{1,n}$, which means that, for every $i\in \overline{1,n}$, there exists an $e_i\in {\cal B}(A_i)$ such that $e_i\in [x_i)$ and $\neg \, e_i\in [\neg \, x_i)$, that is $x_i^{j_i}\leq e_i$ and $(\neg \, x_i)^{k_i}\leq \neg \, e_i$ for some $j_i,k_i\in \N ^*$. Let $j=\max \{j_i\ |\ i\in \overline{1,n}\}\in \N ^*$ and $k=\max \{k_i\ |\ i\in \overline{1,n}\}\in \N ^*$. By Lemma \ref{aritmlr}, (\ref{aritmlr1}), it follows that, for all $i\in \overline{1,n}$, $x_i^j\leq e_i$ and $(\neg \, x_i)^k\leq \neg \, e_i$, hence $x^j=(x_1^j,\ldots ,x_n^j)\leq (e_1,\ldots ,e_n)$ and $(\neg \, x)^k=((\neg \, x_1)^k,\ldots ,(\neg \, x_n)^k)\leq (\neg \, e_1,\ldots ,\neg \, e_n)=\neg \, (e_1,\ldots ,e_n)$. Let $e=(e_1,\ldots ,e_n)\in {\cal B}(A)$. We have shown that $x^j\leq e$ and $(\neg \, x)^k\leq \neg \, e$, with $j,k\in \N ^*$, which means that $e\in [x)$ and $\neg \, e\in [\neg \, x)$, hence $x\in S(A)$, therefore $\displaystyle \prod _{i=1}^nS(A_i)\subseteq S(A)$. Thus $\displaystyle S(A)=S(\prod _{i=1}^nA_i)=\prod _{i=1}^nS(A_i)$.\end{proof}

\begin{lemma}
Let $(A_i)_{i\in I}$ be a non--empty family of residuated lattices all of which have $\odot =\wedge $, and $\displaystyle A=\prod _{i\in I}A_i$. Then $\displaystyle S(A)=S(\prod _{i\in I}A_i)=\prod _{i\in I}S(A_i)$.\label{proddirl2}\end{lemma}

\begin{proof} The fact that $\displaystyle S(A)=S(\prod _{i\in I}A_i)\subseteq \prod _{i\in I}S(A_i)$ follows from Lemma \ref{proddirl0}. Now let $\displaystyle x=(x_i)_{i\in I}\in \prod _{i\in I}S(A_i)$, that is $x_i\in S(A_i)$ for each $i\in I$, which means that, for every $i\in I$, there exists an $e_i\in {\cal B}(A_i)$ such that $e_i\in [x_i)$ and $\neg \, e_i\in [\neg \, x_i)$, that is $x_i\leq e_i$ and $\neg \, x_i\leq \neg \, e_i$, since every element of these residuated lattices is idempotent. Then $x\leq e$ and $\neg \, x\leq \neg \, e$, where $\displaystyle e=(e_i)_{i\in I}\in \prod _{i\in I}{\cal B}(A_i)={\cal B}(A)$. Hence $e\in [x)$ and $\neg \, e\in [\neg \, x)$, therefore $x\in S(A)$, so the inclusion $\displaystyle \prod _{i\in I}S(A_i)\subseteq S(A)$ holds as well. Thus $\displaystyle S(A)=S(\prod _{i\in I}A_i)=\prod _{i\in I}S(A_i)$ in this particular case.\end{proof}

\begin{proposition}
Let $n\in \N ^*$, $(A_i)_{i=1}^n$ be a family of residuated lattices, and $\displaystyle A=\prod _{i=1}^nA_i$. Then the following are equivalent:

\begin{enumerate}
\item\label{prodblp1} $A$ has BLP;
\item\label{prodblp2} for every $i\in \overline{1,n}$, $A_i$ has BLP.\end{enumerate}\label{prodblp}\end{proposition}

\begin{proof} From Lemma \ref{proddirl1} and Proposition \ref{caractlrblp} we obtain the equivalences: $A$ has BLP iff $S(A)=A$ iff $\displaystyle S(\prod _{i=1}^nA_i)=\prod _{i=1}^nA_i$ iff $\displaystyle \prod _{i=1}^nS(A_i)=\prod _{i=1}^nA_i$ iff $S(A_i)=A_i$ for every $i\in \overline{1,n}$ iff $A_i$ has BLP for every $i\in \overline{1,n}$.\end{proof}

\begin{proposition}
Let $(A_i)_{i\in I}$ be a non--empty family of residuated lattices all having $\odot =\wedge $, and $\displaystyle A=\prod _{i\in I}A_i$. Then the following are equivalent:

\begin{enumerate}

\item\label{idempprodblp1} $A$ has BLP;
\item\label{idempprodblp2} for every $i\in I$, $A_i$ has BLP.\end{enumerate}\label{idempprodblp}\end{proposition}

\begin{proof} From Lemma \ref{proddirl2} and Proposition \ref{caractlrblp} we obtain the equivalences: $A$ has BLP iff $S(A)=A$ iff $\displaystyle S(\prod _{i\in I}A_i)=\prod _{i\in I}A_i$ iff $\displaystyle \prod _{i\in I}S(A_i)=\prod _{i\in I}A_i$ iff $S(A_i)=A_i$ for every $i\in I$ iff $A_i$ has BLP for every $i\in I$.\end{proof}

\begin{proposition}
Let $(A_i)_{i\in I}$ be a non--empty family of residuated lattices, and $\displaystyle A=\prod _{i\in I}A_i$. If $A$ has BLP, then, for every $i\in I$, $A_i$ has BLP.\label{prodblprez}\end{proposition}

\begin{proof} From Lemma \ref{proddirl0} and Proposition \ref{caractlrblp} we obtain: if $A$ has BLP, then $S(A)=A$, thus $\displaystyle \prod _{i\in I}A_i=A=S(A)=S(\prod _{i\in I}A_i)\subseteq \prod _{i\in I}S(A_i)\subseteq \prod _{i\in I}A_i$, hence $\displaystyle \prod _{i\in I}S(A_i)=\prod _{i\in I}A_i$, therefore $S(A_i)=A_i$ for every $i\in I$, so $A_i$ has BLP for every $i\in I$.\end{proof}

Proposition \ref{prodblp} could have been deduced from the following stronger result which holds in this finite case:

\begin{proposition}
Let $n\in \N ^*$, $(A_i)_{i=1}^n$ be a (finite non--empty) family of residuated lattices, $\displaystyle A=\prod _{i=1}^nA_i$ and, for each $i\in \overline{1,n}$, let $pr_i:A\rightarrow A_i$ be the canonical projection. Let $F$ be an arbitrary filter of $A$ and, for every $i\in \overline{1,n}$, let us denote by $F_i=pr_i(F)$ (which, as we know from Section \ref{preliminaries}, is a filter of $A_i$). Then:

\begin{enumerate}
\item\label{strongprodblp0} for each $i\in \overline{1,n}$, the Boolean morphism ${\cal B}(pr_i):{\cal B}(A)\rightarrow {\cal B}(A_i)$ is surjective;
\item\label{strongprodblp1} the residuated lattices $A/F$ and $\displaystyle \prod _{i=1}^nA_i/F_i$ are isomorphic;
\item\label{strongprodblp2} $F$ has BLP (in $A$) iff $F_i$ has BLP (in $A_i$) for every $i\in \overline{1,n}$;
\item\label{strongprodblp3} the Boolean morphism ${\cal B}(p_F)$ is injective iff the Boolean morphism ${\cal B}(p_{F_i})$ is injective for every $i\in \overline{1,n}$.\end{enumerate}\label{strongprodblp}\end{proposition}

\begin{proof} (\ref{strongprodblp0}) $\displaystyle A=\prod _{i=1}^nA_i$, hence $\displaystyle {\cal B}(A)=\prod _{i=1}^n{\cal B}(A_i)$, thus, for each $i\in \overline{1,n}$, ${\cal B}(pr_i)$ is the canonical projection from the Boolean algebra ${\cal B}(A)$ to the Boolean algebra ${\cal B}(A_i)$, which is a surjective Boolean morphism.

\noindent (\ref{strongprodblp1}) For all $x\in A$ and all $i\in \overline{1,n}$, we shall denote $x_i=pr_i(x)\in A_i$; so $x=(x_1,\ldots ,x_n)$ for all $x\in A$. Let us define $\displaystyle \psi :A/F\rightarrow \prod _{i=1}^nA_i/F_i$ by: for every $x\in A$, $\psi (x/F)=(x_1/F_1,\ldots ,x_n/F_n)$. Then, since $\displaystyle F=\prod _{i=1}^nF_i$ and, for every $x,y\in A$, $x\leftrightarrow y=(x_1\leftrightarrow y_1,\ldots ,x_n\leftrightarrow y_n)$, it follows that $\psi $ is well defined and injective, because, for all $x,y\in A$, these equivalences hold: $x/F=y/F$ iff $x\leftrightarrow y\in F$ iff $\displaystyle (x_1\leftrightarrow y_1,\ldots ,x_n\leftrightarrow y_n)\in \prod _{i=1}^nF_i$ iff, for each $i\in \overline{1,n}$, $x_i\leftrightarrow y_i\in F_i$ iff, for each $i\in \overline{1,n}$, $x_i/F_i=y_i/F_i$ iff $(x_1/F_1,\ldots ,x_n/F_n)=(y_1/F_1,\ldots ,y_n/F_n)$. Clearly, $\psi $ is a surjective morphism of residuated lattices. Thus $\psi $ is a residuated lattice isomorphism.

\noindent (\ref{strongprodblp2}) ``$\Rightarrow $``: According to (\ref{strongprodblp1}), $\displaystyle \psi :A/F\rightarrow \prod _{i=1}^nA_i/F_i$ is a residuated lattice isomorphism, hence $\displaystyle {\cal B}(\psi ):{\cal B}(A/F)\rightarrow {\cal B}(\prod _{i=1}^nA_i/F_i)$ is a Boolean isomorphism. Let us denote, for every $k\in \overline{1,n}$, by $\displaystyle pr^{\prime }_k:\prod _{i=1}^nA_i/F_i\rightarrow A_k/F_k$ the canonical projection. Then, for every $i\in \overline{1,n}$, $\displaystyle pr^{\prime }_i\circ \psi :A/F\rightarrow A_i/F_i$ is a residuated lattice morphism, hence ${\cal B}(pr^{\prime }_i\circ \psi )={\cal B}(pr^{\prime }_i)\circ {\cal B}(\psi ):A/F\rightarrow A_i/F_i$ is a Boolean morphism; moreover, it is a surjective Boolean morphism, because ${\cal B}(\psi )$ is a Boolean isomorphism and ${\cal B}(pr^{\prime }_i)$ is a surjective Boolean morphism, according to (\ref{strongprodblp0}) applied to $\displaystyle \prod _{j=1}^nA_j/F_j$ and $A_i/F_i$ instead of $A$ and $A_i$, respectively. In this implication, $F$ has BLP, thus ${\cal B}(p_F)$ is a surjective Boolean morphism. Therefore ${\cal B}(p_F)\circ {\cal B}(pr^{\prime }_i\circ \psi )$ is a surjective Boolean morphism, for any $i\in \overline{1,n}$.

Let $i\in \overline{1,n}$, arbitrary. We have the following commutative diagram in the category of Boolean algebras:

\begin{center}
\begin{picture}(240,67)(0,0)
\put(50,50){${\cal B}(A)$}
\put(50,5){${\cal B}(A_i)$}
\put(120,50){${\cal B}(A/F)$}
\put(120,5){${\cal B}(A_i/F_i)$}
\put(74,53){\vector(1,0){43}}
\put(77,8){\vector(1,0){40}}
\put(83,57){${\cal B}(p_F)$}
\put(83,13){${\cal B}(p_{F_i})$}
\put(60,46){\vector(0,-1){30}}
\put(135,46){\vector(0,-1){30}}
\put(31,29){${\cal B}(pr_i)$}
\put(138,29){${\cal B}(pr^{\prime }_i)\circ {\cal B}(\psi )$}
\end{picture}
\end{center}

So ${\cal B}(p_{F_i})\circ {\cal B}(pr_i)={\cal B}(p_F)\circ {\cal B}(pr^{\prime }_i\circ \psi )$, which is surjective, hence ${\cal B}(p_{F_i})$ is surjective, which means that $F_i$ has BLP.

\noindent ``$\Leftarrow $``: In this implication, the hypothesis is that, for all $i\in \overline{1,n}$, $F_i$ has BLP, that is ${\cal B}(A_i/F_i)={\cal B}(A_i)/F_i$. Let $x\in A$, such that $x/F\in {\cal B}(A/F)$. Then, keeping all the notations above, $\displaystyle {\cal B}(\psi )(x/F)=\psi (x/F)=(x_1/F_1,\ldots ,x_n/F_n)\in \prod _{i=1}^n{\cal B}(A_i/F_i)=\prod _{i=1}^n{\cal B}(A_i)/F_i$, thus, for all $i\in \overline{1,n}$, $x_i/F_i\in {\cal B}(A_i)/F_i$, so there exists $e_i\in {\cal B}(A_i)$ such that $x_i/F_i=e_i/F_i$. Now the injectivity of $\psi $ shows that $x/F=e/F$, where $\displaystyle e=(e_1,\ldots ,e_n)\in \prod _{i=1}^n{\cal B}(A_i)={\cal B}(A)$, therefore $x/F=e/F\in {\cal B}(A)/F$, hence ${\cal B}(A/F)\subseteq {\cal B}(A)/F$, thus ${\cal B}(A/F)={\cal B}(A)/F$ by Lemma \ref{lifted}, (\ref{lifted4}), which means that $F$ has BLP.

\noindent (\ref{strongprodblp3}) Since $\displaystyle {\cal B}(A)=\prod _{i=1}^n{\cal B}(A_i)$ and $\displaystyle F=\prod _{i=1}^nF_i$, the following equivalences hold, according to Proposition \ref{filtcuinj}: ${\cal B}(p_F)$ is injective iff ${\cal B}(A)\cap F=\{1\}$ iff, for all $i\in \overline{1,n}$, ${\cal B}(A_i)\cap F_i=\{1\}$ iff, for all $i\in \overline{1,n}$, ${\cal B}(p_{F_i})$ is injective.\end{proof}

\begin{remark}
Regarding the proposition above, one may ask to what extent it can be generalized to arbitrary non--empty families of residuated lattices. A quick look at its proof reveals that the following hold in the general case: if $(A_i)_{i\in I}$ is a non--empty family of residuated lattices, $\displaystyle A=\prod _{i\in I}A_i$, for all $i\in I$, $pr_i:A\rightarrow A_i$ is the canonical projection, $F$ is a filter of $A$, for all $i\in I$, $pr_i(F)=F_i$, which is a filter of $A_i$, and we define $\displaystyle \psi :A/F\rightarrow \prod _{i\in I}A_i/F_i$ by: for all $\displaystyle x=(x_i)_{i\in I}\in A=\prod _{i\in I}A_i$ (with $x_i\in A_i$ for each $i\in I$), $\psi (x/F)=(x_i/F_i)_{i\in I}$, then:

\begin{enumerate}
\item for each $i\in I$, the Boolean morphism ${\cal B}(pr_i):{\cal B}(A)\rightarrow {\cal B}(A_i)$ is surjective;
\item $\psi $ is well defined and it is a surjective residuated lattice morphism;
\item if, for every $i\in I$, $F_i$ has BLP, then the Boolean morphism ${\cal B}(\psi )$ is surjective;
\item if, for every $i\in I$, ${\cal B}(p_{F_i})$ is injective, then ${\cal B}(p_F)$ is injective.\end{enumerate}\label{togen}\end{remark}

\begin{corollary}
Let $A$ be an arbitrary residuated lattice. Then, for every $e\in {\cal B}(A)$, the following are equivalent:

\begin{enumerate}
\item $A$ has BLP;
\item the residuated lattices $[e)$ and $[\neg \, e)$ have BLP.
\end{enumerate}\end{corollary}

\begin{proof} Let $e\in {\cal B}(A)$. By Lemma \ref{aritmcbool}, (\ref{aritmcbool1}), $e\wedge \neg \, e=0$ and $e\vee \neg \, e=1$. Then \cite[Proposition $2.18$]{eu3} ensures us that $A$ is isomorphic to the direct product $[e)\times [\neg \, e)$. Now apply Proposition \ref{prodblp}.\end{proof}

\section{Classes of Residuated Lattices with Boolean Lifting Property}
\label{clsBLP}

In this section we introduce and study the conditions $(\star )$ and $(\star \star )$, which turn out to be a strengthening and a weakening of the BLP, respectively, and which open new ways of approaching the study of the BLP. These two conditions share some properties with the BLP (such as Propositions \ref{p1star}, \ref{p2star} and \ref{p2starbis} below) and appear to also differ by some properties from the BLP. They do, however, coincide with the BLP in some remarkable particular cases such as the case when $\odot =\wedge $. This section also contains representation theorems for semilocal and maximal residuated lattices with BLP, and a proof for the fact that local residuated lattices coincide with quasi--local residuated lattices whose Boolean center is equal to $\{0,1\}$. Many of the theoretical results in this section turn out to produce surprising and useful applications. 

Throughout this section, unless mentioned otherwise, $A$ will be an arbitrary residuated lattice.

Let us consider the following conditions, whose notations we shall keep in what follows:

\begin{center}
\begin{tabular}{rl}

$(\star )$ & for all $x\in A$, there exist $u\in {\rm Rad}(A)$ and $e\in {\cal B}(A)$ such that $[x)=[u)\vee [e)$;\\ 
$(\star \star )$ & for all $x\in A$, there exist $u\in A$ and $e\in {\cal B}(A)$ such that $\neg \, u\in N(A)$ and $[x)=[u)\vee [e)$.
\end{tabular}
\end{center}

Clearly, the trivial residuated lattice satisfies $(\star )$.

\begin{remark}
Corollary \ref{radnegnil} shows that $(\star )$ implies $(\star \star )$, that is: if a residuated lattice $A$ satisfies $(\star )$, then $A$ satisfies $(\star \star )$.\label{rem1star}\end{remark}

\begin{remark}
If $A$ has all elements idempotent (that is if $\odot =\wedge $ in $A$), then conditions $(\star )$ and $(\star \star )$ are equivalent in $A$, that is: $A$ satisfies $(\star )$ iff $A$ satisfies $(\star \star )$. The direct implication has been proven in Remark \ref{rem1star}. For the converse implication, just notice that, as shown by Lemma \ref{radunit}, idempotent elements whose negation is nilpotent belong to ${\rm Rad}(A)$.\label{rem2star}\end{remark}

\begin{proposition}
\begin{enumerate}
\item\label{propstar1} $(\star )\Rightarrow BLP$, that is: if $A$ satisfies $(\star )$, then $A$ has BLP. 

\item\label{propstar2} $BLP\Rightarrow (\star \star )$, that is: if $A$ has BLP, then $A$ satisfies $(\star \star )$.

\item\label{propstar3} If $A$ has $\odot =\wedge $, then $(\star )\Leftrightarrow BLP\Leftrightarrow (\star \star )$ in $A$, that is: $A$ satisfies $(\star )$ iff $A$ has BLP iff $A$ satisfies $(\star \star )$.\end{enumerate}\label{propstar}\end{proposition}

\begin{proof} (\ref{propstar1}) Assume that $A$ satisfies $\star $ and let $x\in A$, arbitrary. Then there exist $u\in {\rm Rad}(A)$ and $e\in {\cal B}(A)$ such that $[x)=[u)\vee [e)=[u\odot e)$, thus $u\odot e\in [x)$ and $x\in [u\odot e)$, that is there exist $m,n\in \N ^*$ such that $x^m\leq u\odot e$ and $u^n\odot e=u^n\odot e^n=(u\odot e)^n\leq x$; we have applied Lemma \ref{aritmcbool}, (\ref{aritmcbool2}). Since $x^m\leq u\odot e\leq e$, we have $e\in [x)$. By Lemma \ref{aritmlr}, (\ref{aritmlr6}) and (\ref{aritmlr12}), and Lemma \ref{aritmcbool}, (\ref{aritmcbool4}), $u^n\odot e\leq x$ implies $\neg \, x\leq \neg \, (u^n\odot e)=\neg \, \neg \, e\rightarrow \neg \, u^n=\neg \, e\vee \neg \, u^n$. Since $u\in {\rm Rad}(A)$, by Lemma \ref{radunit}, there exists $k_n\in \N ^*$ such that $(\neg \, u^n)^{k_n}=0$. Therefore, by Lemma \ref{aritmlr}, (\ref{aritmlr1}) and (\ref{aritmlr7}), the distributivity of $\odot $ with respect to $\vee $, Lemma \ref{aritmcbool}, (\ref{aritmcbool1}) and (\ref{aritmcbool2}), and the choice of $k_n$, $e\odot (\neg \, x)^{k_n}\leq e\odot (\neg \, e\vee \neg \, u^n)^{k_n}=e\odot [(\neg \, e)^{k_n}\vee ((\neg \, e)^{k_n-1}\odot \neg \, u^n)\vee \ldots \vee (\neg \, e\odot (\neg \, u^n)^{k_n-1})\vee (\neg \, u^n)^{k_n}]=e\odot [\neg \, e\vee (\neg \, e\odot \neg \, u^n)\vee \ldots \vee (\neg \, e\odot (\neg \, u^n)^{k_n-1})\vee 0]=e\odot \neg \, e\odot [1\vee \neg \, u^n\vee \ldots \vee (\neg \, u^n)^{k_n-1})]=0\odot 1=0$, hence $e\odot (\neg \, x)^{k_n}=0$, so $(\neg \, x)^{k_n}\leq \neg \, e$, hence $\neg \, e\in [\neg \, x)$. We have obtained: $e\in [x)$, $\neg \, e\in [\neg \, x)$ and $e\in {\cal B}(A)$, that is $A$ has BLP by Proposition \ref{caractlrblp}.

\noindent (\ref{propstar2}) Assume that $A$ has BLP and let $x\in A$, arbitrary. By Proposition \ref{caractlrblp}, it follows that there exists an $e\in {\cal B}(A)$ such that $e\in [x)$ and $\neg \, e\in [\neg \, x)$, so there exist $m,n\in \N ^*$ such that $x^n\leq e$ and $(\neg \, x)^m\leq \neg \, e$, thus $e\odot (\neg \, x)^m=0$ by Lemma \ref{aritmlr}, (\ref{aritmlr7}). Let $u=x\vee \neg \, e$. Then, by Lemma \ref{aritmlr}, (\ref{aritmlr13}), and Lemma \ref{aritmcbool}, (\ref{aritmcbool1}) and (\ref{aritmcbool4}), $\neg \, u=\neg \, (x\vee \neg \, e)=\neg \, x\wedge \neg \, \neg \, e=\neg \, x\wedge e=\neg \, x\odot e$. Hence, by Lemma \ref{aritmcbool}, (\ref{aritmcbool2}), $(\neg \, u)^m=(\neg \, x\odot e)^m=(\neg \, x)^m\odot e^m=(\neg \, x)^m\odot e=0$, so $\neg \, u\in N(A)$. By the distributivity of $\odot $ with respect to $\vee $ and Lemma \ref{aritmlr}, (\ref{aritmlr7}), $u\odot e=(x\vee \neg \, e)\odot e=(x\odot e)\vee (\neg \, e\odot e)=(x\odot e)\vee 0=x\odot e$. $x^n\leq e$ (see above) and $x^n\leq x$ by Lemma \ref{aritmlr}, (\ref{aritmlr1}), hence $x^n\leq x\wedge e=x\odot e=u\odot e$ by Lemma \ref{aritmcbool}, (\ref{aritmcbool4}). Thus $x^n\leq u\odot e$, so $u\odot e\in [x)$, thus $[u)\vee [e)=[u\odot e)\subseteq [x)$. $u\odot e=x\odot e\leq x$ by Lemma \ref{aritmlr}, (\ref{aritmlr1}), hence $x\in [u\odot e)$, thus $[x)\subseteq [u\odot e)=[u)\vee [e)$. We have obtained: $[x)=[u)\vee [e)$, $\neg \, u$ is nilpotent and $e\in {\cal B}(A)$; so $A$ satisfies $(\star \star )$.

\noindent (\ref{propstar3}) By (\ref{propstar1}), (\ref{propstar2}) and Remark \ref{rem2star}.\end{proof}

\begin{openproblem} Characterize the class of the residuated lattices in which $(\star )$, BLP and $(\star \star )$ are equivalent.\end{openproblem}

\begin{example}
Not all residuated lattices have the property $(\star \star )$. Indeed, the residuated lattice in Example \ref{exfarablp} has $\odot =\wedge $ and does not have BLP, thus it does not satisfy $(\star \star )$, according to Proposition \ref{propstar}, (\ref{propstar3}).\end{example}

\begin{lemma} \begin{enumerate}
\item\label{boolradnil1} If $A={\cal B}(A)\cup {\rm Rad}(A)\cup N(A)$, then $A$ satisfies $(\star )$.

\item\label{boolradnil2} If $A={\rm Rad}(A)\cup H(A)$, where $H(A)$ is the set of the archimedean elements of $A$, then $A$ satisfies $(\star )$.\end{enumerate}\label{boolradnil}\end{lemma}

\begin{proof} (\ref{boolradnil1}) Although, clearly, any Boolean element is archimedean and any nilpotent element is archimedean, that is ${\cal B}(A)\subseteq H(A)$ and $N(A)\subseteq H(A)$, and so this first statement actually follows from the second, we shall provide here a separate proof for the first statement of this lemma.

So let us assume that $A={\cal B}(A)\cup {\rm Rad}(A)\cup N(A)$, and consider an arbitrary element $x\in A$. Then $x\in {\cal B}(A)$ or $x\in {\rm Rad}(A)$ or $x\in N(A)$.

If $x\in {\cal B}(A)$, then let $u=1\in {\rm Rad}(A)$ (because ${\rm Rad}(A)$ is a filter of $A$) and $e=x\in {\cal B}(A)$. We have $[x)=[1\wedge x)=[u\wedge e)=[u)\vee [e)$.

If $x\in {\rm Rad}(A)$, then let $u=x\in {\rm Rad}(A)$ and $e=1\in {\cal B}(A)$. We have $[x)=[x\wedge 1)=[u\wedge e)=[u)\vee [e)$.

Finally, if $x\in N(A)$, then $[x)=[0)=[1\wedge 0)=[1)\vee [0)$, and we have $1\in {\rm Rad}(A)$ and $0\in {\cal B}(A)$.

Therefore $A$ satisfies $(\star )$.

\noindent (\ref{boolradnil2}) Assume that $A={\cal B}(A)\cup {\rm Rad}(A)\cup H(A)$, and let $x$ be an arbitrary element of $A$. Since the case $x\in {\rm Rad}(A)$ has been treated above, it only remains to treat the case when $x\in H(A)$.

So let us assume that $x$ is an archimedean element of $A$, that is $u=x^n\in {\cal B}(A)$ for some $n\in \N ^*$. Then $[x)=[x^n)=[u)=[u\wedge 1)=[u)\vee [1)$, and $1\in {\rm Rad}(A)$.

Thus $A$ satisfies $(\star )$.\end{proof}

\begin{corollary}
Any hyperarchimedean residuated lattice satisfies $(\star )$, but the converse is not true.\label{maitare2}\end{corollary}

\begin{proof} To show that the converse implication is not true, notice that the residuated lattice in Example \ref{coolex} has BLP and $\odot =\wedge $, thus it satisfies $(\star )$ by Proposition \ref{propstar}, but it is not hyperarchimedean, since all of its elements are idempotent and ${\cal B}(A)=\{0,1\}$, which shows that its middle element $a$ is not archimedean.\end{proof}

\begin{corollary}
Any Boolean algebra induces a residuated lattice with the property $(\star )$.\label{maitare0}\end{corollary}

\begin{proposition} Any local residuated lattice satisfies $(\star )$, but the converse is not true. Moreover, there exist residuated lattices which satisfy $(\star )$, but are not semilocal.\label{maitare}\end{proposition}

\begin{proof} Let $A$ be a local residuated lattice. Then $A={\rm Rad}(A)\cup N(A)$, by Proposition \ref{caractloc}. By Lemma \ref{boolradnil}, it follows that $A$ satisfies $(\star )$.

To disprove the converse implication, just take any of the examples in Remark \ref{qlocnotloc}, which are Boolean algebras, thus residuated lattices satisfying $(\star )$, according to Remark \ref{maitare0}, but they are not local residuated lattices.

Now take the example in Remark \ref{qlocnotsloc}, which is a Boolean algebra, thus a residuated lattice satisfying $(\star )$, by Remark \ref{maitare0}, but it is not a semilocal residuated lattice.\end{proof}

\begin{lemma}
Any non--trivial linearly ordered residuated lattice is local (regardless of its exact residuated lattice structure).\label{chainloc}\end{lemma}

\begin{proof} Let $A$ be a non--trivial linearly ordered residuated lattice. Assume by absurdum that $A$ has two distinct maximal filters $M$ and $P$. Then $M\nsubseteq P$, thus there exists an element $x\in M\setminus P$, so, for every $y\in P$, $y\nsubseteq x$. But, since $A$ is a chain, this means that every $y\in P$ satisfies $x<y$, hence $P\subsetneq [x)\subseteq M$, thus $P\subsetneq M$, and this is a contradiction to the maximality of $P$. Therefore $A$ has only one maximal filter (see Lemma \ref{maxsiprime}, (\ref{maxsiprime2})), that is $A$ is a local residuated lattice.\end{proof}

\begin{corollary}
Any linearly ordered residuated lattice satisfies $(\star )$ (regardless of its exact residuated lattice structure).\label{maitare1}\end{corollary}

\begin{note} Proposition \ref{propstar}, (\ref{propstar1}), shows that Corollary \ref{maitare2} strengthens Corollary \ref{hypblp}, Corollary \ref{maitare0} strengthens Remark \ref{boolblp}, Proposition \ref{maitare} strengthens the first statement in Corollary \ref{locblp}, together with Remarks \ref{qlocnotloc} and \ref{qlocnotsloc}, and Corollary \ref{maitare1} strengthens the main statement in Remark \ref{chainblp}.\end{note}

\begin{proposition}
\begin{enumerate}
\item\label{p1star1} $A$ satisfies $(\star )$ iff $A/F$ satisfies $(\star )$ for every filter $F$ of $A$.
\item\label{p1star2} $A$ satisfies $(\star \star )$ iff $A/F$ satisfies $(\star \star )$ for every filter $F$ of $A$.
\end{enumerate}\label{p1star}\end{proposition}

\begin{proof} (\ref{p1star1}): For the direct implication, assume that $A$ satisfies $(\star )$, and let $F$ be a filter of $A$. Let $x\in A$, arbitrary. Then there exist $u\in {\rm Rad}(A)$ and $e\in {\cal B}(A)$ such that $[x)=[u)\vee [e)=[u\odot e)$, that is $x\in [u\odot e)$ and $u\odot e\in [x)$, that is $(u\odot e)^n\leq x$ and $x^k\leq u\odot e$ for some $n,k\in \N ^*$. Then $u/F\in {\rm Rad}(A)/F={\rm Rad}(A/F)$ and $e/F\in {\cal B}(A)/F\subseteq {\cal B}(A/F)$ (see Lemma \ref{lifted}, (\ref{lifted4})), $(u/F\odot e/F)^n\leq x/F$ and $(x/F)^k\leq u/F\odot e/F$, hence $x/F\in [u/F\odot e/F)$ and $u/F\odot e/F\in [x/F)$, that is $[x/F)=[u/F\odot e/F)=[u/F)\vee [e/F)$, therefore $A/F$ satisfies $(\star )$.

For the converse implication, just take $F=\{1\}$.

\noindent (\ref{p1star2}): Same as the proof for (\ref{p1star1}), except, in the argument for the direct implication, we need to notice that, if $u\in A$ such that $\neg \, u\in N(A)$, then clearly $\neg \, u/F\in N(A/F)$.\end{proof}

\begin{remark} By Proposition \ref{propstar}, (\ref{propstar3}), and Proposition \ref{idempprodblp}, if $(A_i)_{i\in I}$ is a non--empty family of residuated lattices all having $\odot =\wedge $ and $\displaystyle A=\prod _{i\in I}A_i$, then: $A$ satisfies $(\star )$ iff $A$ has BLP iff $A$ satisfies $(\star \star )$ iff each $A_i$ satisfies $(\star )$ iff each $A_i$ has BLP iff each $A_i$ satisfies $(\star \star )$.\end{remark}

\begin{proposition}
Let $(A_i)_{i\in I}$ be a non--empty family of residuated lattices and $\displaystyle A=\prod _{i\in I}A_i$. Then:

\begin{enumerate}
\item\label{p2star1} If $A$ satisfies $(\star )$, then $A_i$ satisfies $(\star )$ for each $i\in I$.
\item\label{p2star2} If $A$ satisfies $(\star \star )$, then $A_i$ satisfies $(\star \star )$ for each $i\in I$.
\end{enumerate}\label{p2star}\end{proposition}

\begin{proof} (\ref{p2star1}) Assume that $A$ satisfies $(\star )$, and let $k\in I$ and $x_k\in A_k$, both arbitrary but fixed. For all $i\in I\setminus \{k\}$, let $x_i\in A_i$, arbitrary. Then $x=(x_i)_{i\in I}\in A$. Since $A$ satisfies $(\star )$, it follows that there exist $u=(u_i)_{i\in I}\in {\rm Rad}(A)$ ($u_i\in A_i$ for every $i\in I$) and $\displaystyle e=(e_i) _{i\in I}\in {\cal B}(A)=\prod _{i\in I}{\cal B}(A_i)$ ($e_i\in A_i$ for every $i\in I$), such that $[x)=[u)\vee [e)=[u\odot e)$, that is $x\in [u\odot e)$ and $u\odot e\in [x)$, so $(u\odot e)^p\subseteq x$ and $x^q\subseteq u\odot e$ for some $p,q\in \N ^*$, thus, for every $i\in I$, $(u_i\odot e_i)^p\subseteq x_i$ and $x_i^q\subseteq u_i\odot e_i$, hence $x_i\in [u_i\odot e_i)$ and $u_i\odot e_i\in [x_i)$, which means that $[x_i)=[u_i\odot e_i)=[u_i)\vee [e_i)$. Then, for every $i\in I$, $u_i\in {\rm Rad}(A_i)$ and $e_i\in {\cal B}(A_i)$, hence $u_k\in {\rm Rad}(A_k)$ and $e_k\in {\cal B}(A_k)$. Also, $[x_k)=[u_k)\vee [e_k)$. Hence $A_k$ satisfies $(\star )$.

\noindent (\ref{p2star2}) Similar to the proof of (\ref{p2star1}), once we notice that, if $u=(u_i)_{i\in I}\in A$ has $\neg \, u\in N(A)$, then, for all $i\in I$, $\neg \, u_i\in N(A_i)$.\end{proof}

\begin{proposition}
Let $n\in \N ^*$, $(A_i)_{i=1}^n$ be a family of residuated lattices and $\displaystyle A=\prod _{i=1}^nA_i$. Then:

\begin{enumerate}
\item\label{p2starbis1} $A$ satisfies $(\star )$ iff $A_i$ satisfies $(\star )$ for each $i\in \overline{1,n}$.

\item\label{p2starbis2} $A$ satisfies $(\star \star )$ iff $A_i$ satisfies $(\star \star )$ for each $i\in \overline{1,n}$.
\end{enumerate}\label{p2starbis}\end{proposition}

\begin{proof} (\ref{p2starbis1}) The direct implication follows from Proposition \ref{p2star}, (\ref{p2star1}). For the converse implication, assume that each of the residuated lattices $A_1,\ldots ,A_n$ satisfies $(\star )$, and let $x\in A$, arbitrary. Then, for every $i\in \overline{1,n}$, there exists $u_i\in {\rm Rad}(A_i)$ and $e_i\in {\cal B}(A_i)$, such that $[x_i)=[u_i)\vee [e_i)=[u_i\odot e_i)$, that is $x_i\in [u_i\odot e_i)$ and $u_i\odot e_i\in [x_i)$, so $(u_i\odot e_i)^{m_i}\leq x_i$ and $x_i^{k_i}\leq u_i\odot e_i$ for some $m_i,k_i\in \N ^*$. Let $\displaystyle e=(e_1,\ldots ,e_n)\in \prod _{i=1}^n{\cal B}(A_i)={\cal B}(A)$ and $\displaystyle u=(u_1,\ldots ,u_n)\in \prod _{i=1}^n{\rm Rad}(A_i)={\rm Rad}(A)$. Let $m=\max \{m_1,\ldots ,m_n\}\in \N ^*$ and $k=\max \{k_1,\ldots ,k_n\}\in \N ^*$. Then, by Lemma \ref{aritmlr}, (\ref{aritmlr1}), we get that, for all $i\in \overline{1,n}$, $(u_i\odot e_i)^m\leq x_i$ and $x_i^k\leq u_i\odot e_i$, that is $(u\odot e)^m\leq x$ and $x^k\leq u\odot e$, hence $x\in [u\odot e)$ and $u\odot e\in [x)$, which means that $[x)=[u\odot e)=[u)\vee [e)$. Therefore $A$ satisfies $(\star )$.

\noindent (\ref{p2starbis2}) The direct implication follows from Proposition \ref{p2star}, (\ref{p2star2}). For the converse implication, the proof goes similarly to the one above in (\ref{p2starbis1}), once we notice that, if the negation of each component of an element of $A$ is nilpotent, then the negation of that element of $A$ is nilpotent.\end{proof}

\begin{proposition}
The following are equivalent:

\begin{enumerate}
\item\label{p3star1} $A$ satisfies $(\star )$;
\item\label{p3star2} $A/{\rm Rad}(A)$ satisfies $(\star )$ and ${\rm Rad}(A)$ has BLP (in $A$).\end{enumerate}\label{p3star}\end{proposition}

\begin{proof} (\ref{p3star1})$\Rightarrow $(\ref{p3star2}): By Proposition \ref{p1star}, (\ref{p1star1}), and Proposition \ref{propstar}, (\ref{propstar1}), if $A$ satisfies $(\star )$, then $A/{\rm Rad}(A)$ satisfies $(\star )$, and $A$ has BLP, hence ${\rm Rad}(A)$ has BLP.

\noindent (\ref{p3star2})$\Rightarrow $(\ref{p3star1}): Assume that $A/{\rm Rad}(A)$ satisfies $(\star )$ and ${\rm Rad}(A)$ has BLP, and let $x\in A$, arbitrary. Since $A/{\rm Rad}(A)$ satisfies $(\star )$, it follows that there exists $u\in A$ such that $u/{\rm Rad}(A)\in {\rm Rad}(A/{\rm Rad}(A))=\{1/{\rm Rad}(A)\}$ and there exists $e\in A$ such that $e/{\rm Rad}(A)\in {\cal B}(A/{\rm Rad}(A))={\cal B}(A)/{\rm Rad}(A)$ (because ${\rm Rad}(A)$ has BLP), with the property that $[x/{\rm Rad}(A))=[u/{\rm Rad}(A))\vee [e/{\rm Rad}(A))= [u/{\rm Rad}(A)\odot e/{\rm Rad}(A))$. But then $u/{\rm Rad}(A)=1/{\rm Rad}(A)$ and there exists $f\in {\cal B}(A)$ such that $e/{\rm Rad}(A)=f/{\rm Rad}(A)$. Therefore $[x/{\rm Rad}(A))=[1/{\rm Rad}(A)\odot f/{\rm Rad}(A))=[f/{\rm Rad}(A))$, so $x/{\rm Rad}(A)\in [f/{\rm Rad}(A))$ and $f/{\rm Rad}(A)\in [x/{\rm Rad}(A))$, that is $f/{\rm Rad}(A)=f^m/{\rm Rad}(A)\leq x/{\rm Rad}(A)$ and $x^n/{\rm Rad}(A)\leq f/{\rm Rad}(A)$ for some $m,n\in \N ^*$, so $x^n/{\rm Rad}(A)\leq f/{\rm Rad}(A)=f^n/{\rm Rad}(A)\leq x^n/{\rm Rad}(A)$, thus $x^n/{\rm Rad}(A)=f/{\rm Rad}(A)$, that is $x^n\leftrightarrow f\in {\rm Rad}(A)$, by Lemma \ref{aritmcbool}, (\ref{aritmcbool2}), and Lemma \ref{aritmlr}, (\ref{aritmlr1}). So $x^n\leftrightarrow f=v$, with $v\in {\rm Rad}(A)$. Thus $(x^n\rightarrow f)\wedge (f\rightarrow x^n)=v$, hence $x^n\rightarrow f\leq v$ and $f\rightarrow x^n\leq v$, that is $x^n\leq f\odot v=v\odot f$ and $f\leq x^n\odot v\leq x^n\leq x$, by the law of residuation and Lemma \ref{aritmlr}, (\ref{aritmlr1}). Then $v\odot f\in [x)$, thus $[v)\vee [f)=[v\odot f)\subseteq [x)$, and $f\leq x$, thus $x\in [f)\subseteq [v)\vee [f)$, hence $[x)\subseteq [v)\vee [f)$. Therefore $[x)=[v)\vee [f)$, with $v\in {\rm Rad}(A)$ and $f\in {\cal B}(A)$, so $A$ satisfies $(\star )$.\end{proof}

A subset $\{x_1,\ldots ,x_n\}\subseteq A$, with $n\in \N ^*$, is said to be {\em complete} iff $\displaystyle \bigwedge _{i=1}^nx_i=0$ and $x_i\vee x_j=1$ for all $i,j\in \overline{1,n}$ such that $i\neq j$. Clearly, if $A$ is non--trivial, then any complete subset of $A$ has at least two elements. 

\begin{proposition} The following are equivalent:

\begin{enumerate}
\item\label{semiloc1} $A$ is semilocal and has BLP;
\item\label{semiloc2} $A$ is semilocal and ${\rm Rad}(A)$ has BLP;
\item\label{semiloc3} there exist $n\in \N ^*$ and a complete set $\{e_1,\ldots ,e_n\}\subseteq {\cal B}(A)$ such that $[e_1),\ldots ,[e_n)$ are local residuated lattices;
\item\label{semiloc4} $A$ is isomorphic to a finite direct product of local residuated lattices.
\end{enumerate}\label{semiloc}\end{proposition}

\begin{proof} (\ref{semiloc1})$\Rightarrow $(\ref{semiloc2}): Trivial.

\noindent (\ref{semiloc2})$\Rightarrow $(\ref{semiloc3}): Assume that $A$ is semilocal, so that $|{\rm Max}(A)|=n$ for some $n\in \N ^*$, and assume also that ${\rm Rad}(A)$ has BLP in $A$. Let ${\rm Max}(A)=\{M_1,\ldots ,M_n\}$, so $A/{\rm Rad}(A)$ is isomorphic to $\displaystyle \prod _{i=1}^nA/M_i$, by \cite[Proposition $5.2$]{eu3}. Then, according to \cite[Proposition $2.17$]{eu3}, there exists an $n\in \N ^*$ and a complete subset in $A/{\rm Rad}(A)$ $\{f_1,\ldots ,f_n\}\subseteq {\cal B}(A/{\rm Rad}(A))$ such that, for each $i\in \overline{1,n}$, $[f_i)$ is isomorphic to $A/M_i$. Since ${\rm Rad}(A)$ has BLP, it follows that there exist $e_1,\ldots ,e_n\in {\cal B}(A)$ such that $f_i=e_i/{\rm Rad}(A)$ for all $i\in \overline{1,n}$. For the sake of completeness, here we shall repeat an argument from \cite[Theorem $6.5$]{eu3}. Let $i,j\in \overline{1,n}$ such that $i\neq j$, arbitrary. The fact that $\{f_1,\ldots ,f_n\}$ is a complete subset of $A/{\rm Rad}(A)$ shows that $(e_1\wedge \ldots \wedge e_n)/{\rm Rad}(A)=f_1\wedge \ldots \wedge f_n=0/{\rm Rad}(A)$ and $(e_i\vee e_j)/{\rm Rad}(A)=f_i\vee f_j=1/{\rm Rad}(A)$, thus $\neg \, (e_1\wedge \ldots \wedge e_n)\in {\rm Rad}(A)$ and $e_i\vee e_j\in {\rm Rad}(A)$. But $e_i\vee e_j\in {\cal B}(A)$ and, as shown by Lemma \ref{aritmcbool}, (\ref{aritmcbool1}), $\neg \, (e_1\wedge \ldots \wedge e_n)\in {\cal B}(A)$ also. Now Lemma \ref{lemaradsid}, (\ref{boolsirad}) shows that $e_i\vee e_j=1$ and $\neg \, (e_1\wedge \ldots \wedge e_n)=1$, hence $e_1\wedge \ldots \wedge e_n=\neg \, \neg \, (e_1\wedge \ldots \wedge e_n)=\neg \, 1=0$, according to Lemma \ref{aritmcbool}, (\ref{aritmcbool1}), and Lemma \ref{aritmlr}, (\ref{aritmlr3}). So $\{e_1,\ldots ,e_n\}$ is a complete subset of $A$ formed of Boolean elements of $A$, hence, by \cite[Proposition $2.18$]{eu3}, $A$ is isomorphic to $\displaystyle \prod _{i=1}^n\, [e_i)$, thus $\displaystyle |{\rm Max}(A)|=\sum _{i=1}^n|{\rm Max}([e_i))|$. Remember that, for every $i\in \overline{1,n}$ $[e_i/{\rm Rad}(A))=[f_i)$ is isomorphic to $A/M_i$, which is a non--trivial residuated lattice since $M_i$ is a proper filter of $A$; thus $[e_i/{\rm Rad}(A))$ is a non--trivial residuated lattice, which clearly implies that $[e_i)$ is a non--trivial residuated lattice, therefore $|{\rm Max}([e_i))|\geq 1$ by Lemma \ref{maxsiprime}, (\ref{maxsiprime2}). So $\displaystyle n=\sum _{i=1}^n|{\rm Max}([e_i))|$ and, for each $i\in \overline{1,n}$, $|{\rm Max}([e_i))|\geq 1$, which means that $|{\rm Max}([e_i))|=1$ for all $i\in \overline{1,n}$, that is $[e_1),\ldots ,[e_n)$ are local residuated lattices. 

\noindent (\ref{semiloc3})$\Rightarrow $(\ref{semiloc1}): Let $n\in \N ^*$ and $\{e_1,\ldots ,e_n\}\subseteq {\cal B}(A)$ be a complete subset of $A$. Then, according to \cite[Proposition $2.18$]{eu3}, $A$ is isomorphic to $\displaystyle \prod _{i=1}^n\, [e_i)$, hence $\displaystyle |{\rm Max}(A)|=\sum _{i=1}^n|{\rm Max}([e_i))|$, thus, if $[e_1),\ldots ,[e_n)$ are local residuated lattices, that is $|{\rm Max}([e_1))|=\ldots =|{\rm Max}([e_n))|=1$, then $|{\rm Max}(A)|=n$, so $A$ is a semilocal residuated lattice. Moreover, by Corollary \ref{locblp} and Proposition \ref{prodblp}, it follows that $A$ has BLP.

\noindent (\ref{semiloc3})$\Rightarrow $(\ref{semiloc4}): According to \cite[Proposition $2.18$]{eu3}, if $n\in \N ^*$ and $\{e_1,\ldots ,e_n\}\subseteq {\cal B}(A)$ is a complete set, then $A$ is isomorphic to $\displaystyle \prod _{i=1}^n\, [e_i)$, thus if, for each $i\in \overline{1,n}$, $[e_i)$ is a local residuated lattice, it follows that $A$ is isomorphic to a finite direct product of local residuated lattices.

\noindent (\ref{semiloc4})$\Rightarrow $(\ref{semiloc3}): According to \cite[Proposition $2.17$]{eu3}, if $n\in \N ^*$ and $A_1,\ldots ,A_n$ are residuated lattices such that $A$ is isomorphic to $\displaystyle \prod _{i=1}^nA_i$, then there exists a complete set $\{e_1,\ldots ,e_n\}\subseteq {\cal B}(A)$ such that, for each $i\in \overline{1,n}$, $A_i$ is isomorphic to $[e_i)$, thus, if $A_i$ is local, then $[e_i)$ is a local residuated lattice as well.\end{proof}

Adopting a terminology from ring theory, we say that a residuated lattice $A$ is {\em semiperfect} iff it satisfies the equivalent conditions from Proposition \ref{semiloc}.

\begin{corollary}
Any semiperfect residuated lattice has BLP.\label{sperfect}\end{corollary}

\begin{corollary} If $A$ is finite, then: $A$ has BLP iff ${\rm Rad}(A)$ has BLP.\label{lafinite}\end{corollary}

\begin{corollary}
If ${\rm  Rad}(A)$ has BLP, then the following are equivalent:

\begin{enumerate}
\item $A$ is semilocal;
\item $A$ is isomorphic to a finite direct product of local residuated lattices.
\end{enumerate}\label{corsemiloc}\end{corollary}

\begin{corollary}
The following are equivalent:

\begin{enumerate}
\item $A$ is a semilocal MV--algebra (respectively a semilocal BL--algebra);
\item $A$ is isomorphic to a finite direct product of local MV--algebras (respectively BL--algebras).
\end{enumerate}\end{corollary}

\begin{proof} According to \cite[Proposition $5$]{figele}, any MV--algebra has BLP. According to \cite[Lemma $2.7.6$]{leo}, any BL--algebra has BLP. Clearly, if an MV--algebra (respectively a BL--algebra) is isomorphic to a finite direct product of residuated lattices, then each of those residuated lattices is an MV--algebra (respectively a BL--algebra). Now apply Corollary \ref{corsemiloc}.\end{proof}

\begin{corollary}
If $A$ is semilocal and ${\rm Rad}(A)$ has BLP, then $A$ satisfies $(\star )$.\label{cortare}\end{corollary}

\begin{proof} By Corollary \ref{corsemiloc}, Corollary \ref{maitare1} and Proposition \ref{p2starbis}.\end{proof}

We can rephrase the previous corollary like this:

\begin{corollary}
Any semiperfect residuated lattice satisfies $(\star )$.\label{totcortare}\end{corollary}

\begin{corollary}
Any residuated lattice $A$ with BLP which does not satisfy $(\star )$ has ${\rm Max}(A)$ infinite. Equivalently: if a residuated lattice $A$ is such that ${\rm Rad}(A)$ has BLP and $A$ does not satisfy $(\star )$, then $A$ has ${\rm Max}(A)$ infinite.\label{mymy}\end{corollary}

\begin{corollary}
If $A$ is finite and it has BLP, then $A$ satisfies $(\star )$. Equivalently: if $A$ is finite and ${\rm Rad}(A)$ has BLP, then $A$ satisfies $(\star )$.\label{wowwow}\end{corollary}

The previous results suggest that solving the following open problem may prove difficult.

\begin{openproblem} \begin{enumerate}
\item Find a residuated lattice with BLP that does not satisfy $(\star )$.
\item Find a residuated lattice without BLP that satisfies $(\star \star )$.\end{enumerate}\end{openproblem}

$A$ is called a {\em maximal residuated lattice} iff, for any index set $I$, any family $(a_i)_{i\in I}\subseteq A$ and any family $(F_i)_{i\in I}\subseteq {\cal F}(A)$, if these families have the property that, given any finite subset $J$ of $I$, there exists $x_J\in A$ such that $x_J/F_i=a_i/F_i$ for all $i\in J$, then there exists $x\in A$ such that $x/F_i=a_i/F_i$ for all $i\in I$. Clearly, all residuated lattices $A$ with ${\cal F}(A)$ finite are maximal (the converse is known to be false \cite{figele}), thus finite residuated lattices are maximal, and simple residuated lattices are maximal.

\begin{proposition}
The following are equivalent:

\begin{enumerate}
\item\label{maximal0} $A$ is maximal and has BLP;
\item\label{maximal1} $A$ is maximal and ${\rm Rad}(A)$ has BLP;
\item\label{maximal2} $A$ is isomorphic to a finite direct product of local maximal residuated lattices.
\end{enumerate}\label{maximal}\label{resultsflow}\end{proposition}

\begin{proof} (\ref{maximal0})$\Rightarrow $(\ref{maximal1}): Trivial.

\noindent (\ref{maximal1})$\Rightarrow $(\ref{maximal2}): According to \cite[Proposition $6.2$]{eu3}, any maximal residuated lattice is semilocal. According to \cite[Proposition $6.3$]{eu3}, if $A$ is maximal and $e\in {\cal B}(A)$, then $[e)$ is a maximal residuated lattice. Now apply Proposition \ref{semiloc}.

\noindent (\ref{maximal2})$\Rightarrow $(\ref{maximal0}): According to \cite[Proposition $6.4$]{eu3}, a finite direct product of maximal residuated lattices is a maximal residuated lattice. By Proposition \ref{semiloc}, a finite direct product of local residuated lattices has BLP.\end{proof}

\begin{corollary}
If ${\rm Rad}(A)$ has BLP, then the following are equivalent:

\begin{enumerate}
\item $A$ is maximal;
\item $A$ is isomorphic to a finite direct product of local maximal residuated lattices.
\end{enumerate}\label{furtherflow}\end{corollary}

\begin{proposition} The following are equivalent:

\begin{enumerate}
\item\label{locprop1} $A$ is local;
\item\label{locprop0} $A$ is quasi--local and ${\cal B}(A)=\{0,1\}$; 
\item\label{locprop2} $A$ has BLP and ${\cal B}(A)=\{0,1\}$;
\item\label{locprop3} $A$ satisfies $(\star )$ and ${\cal B}(A)=\{0,1\}$;
\item\label{locprop4} $A=N(A)\cup \{x\in A\ |\ \neg \, x\in N(A)\}$ and ${\cal B}(A)=\{0,1\}$.
\end{enumerate}\label{locprop}\end{proposition}

\begin{proof} (\ref{locprop1})$\Rightarrow $(\ref{locprop3}): By Proposition \ref{maitare} and Lemma \ref{implicdir}.

\noindent (\ref{locprop3})$\Rightarrow $(\ref{locprop2}): By Proposition \ref{propstar}, (\ref{propstar1}).

\noindent (\ref{locprop2})$\Rightarrow $(\ref{locprop4}): By Remark \ref{booltriv} and Proposition \ref{caractlrblp}.

\noindent (\ref{locprop4})$\Rightarrow $(\ref{locprop2}): By Lemma \ref{boolradnil}.

\noindent (\ref{locprop2})$\Leftrightarrow $(\ref{locprop0}): By Proposition \ref{qlocblp}.

\noindent (\ref{locprop0})$\Rightarrow $(\ref{locprop1}): Assume that $A$ is quasi--local and ${\cal B}(A)=\{0,1\}$. Let $x,y\in A$, such that $x\odot y\in N(A)$, that is $[x)\vee [y)=[x\odot y)=A$. Then, according to Proposition \ref{qlocblp}, it follows that there exist $e,f\in {\cal B}(A)=\{0,1\}$ such that $e\vee f=1$ (thus $e=1$ or $f=1$) and $[x\odot e)=[x)\vee [e)=[x\odot f)=[y)\vee [f)=A$. If $e=1$, then $[x)=[x\odot 1)=A$, thus $x\in N(A)$; if $f=1$, then $[y)=[y\odot 1)=A$, thus $y\in N(A)$. By Proposition \ref{caractloc}, it follows that $A$ is local.\end{proof}

\begin{corollary} \begin{enumerate}
\item\label{splendid1} If ${\cal B}(A)=\{0,1\}$, then: $A$ satisfies $(\star )$ iff $A$ has BLP iff $A$ is local.
\item\label{splendid2} If $\odot =\wedge $ and ${\cal B}(A)=\{0,1\}$, then: $A$ satisfies $(\star )$ iff $A$ has BLP iff $A$ satisfies $(\star \star )$ iff $A$ is local.\end{enumerate}\label{splendid}\end{corollary}

\begin{proof} (\ref{splendid1}) By Proposition \ref{locprop}.

\noindent (\ref{splendid2}) By (\ref{splendid1}) and Proposition \ref{propstar}, (\ref{propstar3}).\end{proof}

\begin{remark} Corollary \ref{splendid} provides us with a quite productive method to obtain residuated lattices which satisfy $(\star )$/BLP/$(\star \star )$ and residuated lattices which do not satisfy $(\star )$/BLP/$(\star \star )$. Each of the following types of residuated lattice has the Boolean center equal to $\{0,1\}$, which leads to the results below:

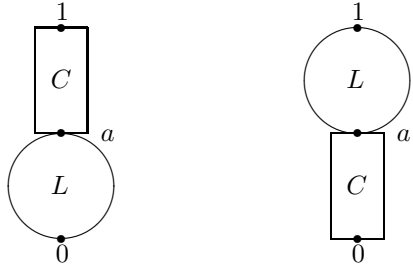
\begin{center}
\begin{tabular}{cc}
\begin{picture}(100,110)(0,0)
\put(50,15){\circle*{3}}
\put(50,55){\circle*{3}}
\put(50,95){\circle*{3}}
\put(48,6){$0$}
\put(48,98){$1$}
\put(65,52){$a$}
\put(46,72){$C$}
\put(46,32){$L$}
\put(40,55){\line(1,0){20}}
\put(40,95){\line(1,0){20}}
\put(40,55){\line(0,1){40}}
\put(60,55){\line(0,1){40}}
\put(50,35){\circle{40}}
\end{picture}
&
\begin{picture}(100,110)(0,0)
\put(50,15){\circle*{3}}
\put(50,55){\circle*{3}}
\put(50,95){\circle*{3}}
\put(48,6){$0$}
\put(48,98){$1$}
\put(65,52){$a$}
\put(46,72){$L$}
\put(46,32){$C$}
\put(40,15){\line(1,0){20}}
\put(40,55){\line(1,0){20}}
\put(40,15){\line(0,1){40}}
\put(60,15){\line(0,1){40}}
\put(50,75){\circle{40}}
\end{picture}
\end{tabular}
\end{center}

\begin{enumerate}
\item If a residuated lattice $A$ is such that its underlying bounded lattice is formed of a non--trivial bounded lattice $L$, with first element $0$ and last element $a\in A$, and a bounded chain $C$, with first element $a$ and last element $1$, as suggested in the picture above in the left, then it is straightforward that, by considering the restrictions of the operations $\vee $, $\wedge $ and $\odot $ of $A$ to $L$, and the binary operation $\leadsto $ on $L$ defined by: for all $x,y\in L$, $x\leadsto y=\begin{cases}a, & \mbox{if }x\leq y,\\ x\rightarrow y, & \mbox{otherwise},\end{cases}$ then $(L,\vee ,\wedge ,\odot ,\leadsto ,0,a)$ becomes a residuated lattice, and ${\rm Max}(A)=\{C\cup M\ |\ M\in {\rm Max}(L)\}$, thus $A$ is local iff $L$ is local, therefore: $A$ satisfies $(\star )$ iff $A$ has BLP iff $A$ is local iff $L$ is local. If, moreover, $\odot =\wedge $ in $A$ (thus also in $L$), then : $A$ satisfies $(\star )$ iff $A$ has BLP iff $A$ satisfies $(\star \star )$ iff $A$ is local iff $L$ is local (see Example \ref{exfarablp}).
\item If a residuated lattice $A$ is such that its underlying bounded lattice is formed of a non--trivial bounded chain $C$, with first element $0$ and last element $a\in A$, and a bounded lattice $L$, with first element $a$ and last element $1$, as suggested in the picture above in the right, then, similarly to the above, it is straightforward that, by considering the restrictions of the operations $\vee $, $\wedge $ and $\odot $ of $A$ to $C$, and the binary operation $\leadsto $ on $C$ defined by: for all $x,y\in C$, $x\leadsto y=\begin{cases}a, & \mbox{if }x\leq y,\\ x\rightarrow y, & \mbox{otherwise},\end{cases}$ then $(C,\vee ,\wedge ,\odot ,\leadsto ,0,a)$ becomes a residuated lattice, which is local by Lemma \ref{chainloc}, therefore $A$ is local, with the unique maximal filter equal to $L\cup M$, where $M$ is the unique maximal filter of $C$, hence $A$ satisfies $(\star )$, thus $A$ also satisfies BLP and $(\star \star )$.\end{enumerate}\end{remark}

\end{document}